\documentclass[final,3p,times]{elsarticle}

\usepackage{amsmath}
\usepackage{amsfonts}
\usepackage{amssymb}
\usepackage[english]{babel}
\usepackage{color}
\usepackage{graphicx}
\usepackage{lmodern}
\usepackage{amsthm}
\usepackage{mathrsfs}
\usepackage{microtype}
\usepackage{mathscinet}
\usepackage{enumitem}
\usepackage[cal=boondoxo,bb=ams]{mathalfa}
\usepackage{hyperref}
\hypersetup{hidelinks}

\newtheorem{thm}{Theorem}[section]
 \newtheorem{cor}[thm]{Corollary}
 \newtheorem{lem}[thm]{Lemma}
 \newtheorem{prop}[thm]{Proposition}
 \newtheorem{defn}[thm]{Definition}
 \newtheorem{rem}[thm]{Remark}

\DeclareMathOperator\supp{supp}

\DeclareMathOperator\meas{meas}

\DeclareMathOperator\loc{loc}

\begin{document}

\begin{frontmatter}



\title{Blow-up for a weakly coupled system of semilinear damped wave equations in the scattering case with power nonlinearities}


\author[Frei]{Alessandro Palmieri}
\ead{alessandro.palmieri.math@gmail.com}

\author[Toho]{Hiroyuki Takamura}
\ead{hiroyuki.takamura.a1@tohoku.ac.jp}

\address[Frei]{Institute of Applied Analysis, Faculty for Mathematics and Computer Science, Technical University Bergakademie Freiberg, Pr\"{u}ferstra{\ss}e 9, 09596, Freiberg, Germany}

\address[Toho]{Mathematical Institute, Tohoku University, Aoba, Sendai 980-8578, Japan}

\begin{abstract}


In this work we study the blow-up of solutions of a weakly coupled system of damped semilinear wave equations in the scattering case with power nonlinearities. We apply an iteration method to study both the subcritical case and the critical case. In the subcritical case our approach is based on lower bounds for the space averages of the components of local solutions. In the critical case we use the slicing method  and a couple of auxiliary functions, recently introduced by Wakasa-Yordanov, to modify the definition of the functionals with the introduction of weight terms. In particular, we find as critical curve for the pair $(p,q)$ of the exponents in the nonlinear terms the same one as for the weakly coupled system of semilinear wave equations with power nonlinearities.

\end{abstract}


\end{frontmatter}


\section{Introduction}

%

In this paper we consider a weakly coupled system of semilinear wave equations with time-dependent, scattering producing damping terms and power nonlinearities, namely,
\begin{align}\label{weakly coupled system}
\begin{cases}
u_{tt}-\Delta u +b_1(t)u_t = |v|^p,  & x\in \mathbb{R}^n, \ t>0,  \\
v_{tt}-\Delta v +b_2(t) v_t = |u|^q,  & x\in \mathbb{R}^n, \ t>0, \\
 (u,u_t,v,v_t)(0,x)= (\varepsilon u_0, \varepsilon u_1, \varepsilon v_0, \varepsilon v_1)(x) & x\in \mathbb{R}^n,
\end{cases}
\end{align}
where $b_1,b_2\in \mathcal{C}([0,\infty))\cap L^1([0,\infty))$ are nonnegative functions, $\varepsilon$ is a positive parameter describing the size of initial data and $p,q>1$.  We will prove blow-up results for \eqref{weakly coupled system} both in the subcritical case and in the critical case.

Let us provide now an historical overview on some results, which are strongly related to our model and the motivations that lead us to consider the nonlinear model \eqref{weakly coupled system}.
Recently, the Cauchy problem for the semilinear wave equation with damping in the scattering case 
\begin{align}
\label{sem wave eq scatt}
\begin{cases}
u_{tt}-\Delta u +b(t)u_t = f(u,\partial_t u),  & x\in \mathbb{R}^n, \ t>0,  \\
 (u, u_t)(0,x)= (\varepsilon u_0, \varepsilon u_1)(x) & x\in \mathbb{R}^n,
\end{cases}
\end{align}
 has been studied in \cite{LT18Scatt,WakYor18damp}, \cite{LT18Glass}, \cite{LT18ComNon} in the cases $f(u,\partial_t u)=|u|^p,|\partial_t u|^p,|\partial_t u|^p+|u|^q$ with $p,q>1$, respectively, provided that $b$ is a continuous, nonnegative and summable function. In particular, for the power nonlinearity $|u|^p$, combing the result in the subcritical case from \cite{LT18Scatt} and the result in the critical case from \cite{WakYor18damp}, we see that the range of values of $p$, for which a blow-up result can be proved, is the same as in case of the classical semilinear wave equation with power nonlinearity. Furthermore, in the above cited papers the upper bounds for the lifespan of the solutions are shown to be the same one (that means also the sharp one)  for the classical semilinear  wave model. More precisely, the condition for the exponent $p$ of the semilinear term, that implies the validity of a blow-up result, is $1<p\leqslant p_0(n)$ for $n\geqslant 2$, where $p_0(n)$ denotes the Strauss exponent, i.e., the positive root of the quadratic equation $$(n-1)p^2-(n+1)p-2=0,$$ and $p>1$ in the one dimensional case. This condition on $p$ is equivalent to require 
 \begin{align}
 \frac{1+p^{-1}}{p-1}\geqslant \frac{n-1}{2}\, . \label{equiv cond Strauss exponent}
 \end{align}
 For the corresponding results in the case of semilinear wave equations we refer to the works \cite{John79,Str81,Kato80,Glas81,Glas81B,Sid84,Scha85,LinSog95, Geo97,Tat01,Jiao03,YZ06,Zhou07}
for the proof of Strauss' conjecture and to \cite{Sid84,Lin90,Zhou92,Zhou93,LinSog96,TakWak11,ZH14}
for the proof of the sharp estimates of the lifespan of local in time solutions. 

On the other hand, it is known that for the weakly coupled system of classical wave equations 
\begin{align}\label{weakly coupled system wave eq}
\begin{cases}
u_{tt}-\Delta u  = |v|^p,  & x\in \mathbb{R}^n, \ t>0,  \\
v_{tt}-\Delta v  = |u|^q,  & x\in \mathbb{R}^n, \ t>0, \\
 (u,u_t,v,v_t)(0,x)= (\varepsilon u_0, \varepsilon u_1, \varepsilon v_0, \varepsilon v_1)(x) & x\in \mathbb{R}^n,
\end{cases}
\end{align}
 the critical curve for the pair $(p,q)$ of exponents is given by the cubic relation
\begin{align*}
\max\left\{\frac{p+2+q^{-1}}{pq-1},\frac{q+2+p^{-1}}{pq-1}\right\}=\frac{n-1}{2}.
\end{align*} For further details on the results for \eqref{weakly coupled system wave eq} we refer to \cite{DGM,DelS97,DM,AKT00,KT03,Kur05,GTZ06,KTW12}. So, we see that the study of the weakly coupled system is not just a simple generalization of the result for the single semilinear equation. Indeed, it holds
\begin{align}\label{ineq}
\max\left\{\frac{p+2+q^{-1}}{pq-1},\frac{q+2+p^{-1}}{pq-1}\right\} \geqslant \max\left\{\frac{1+p^{-1}}{p-1},\frac{1+q^{-1}}{q-1}\right\}, 
\end{align} where the equality is satisfied only in the case $p=q$. Therefore, according to \eqref{equiv cond Strauss exponent} and \eqref{ineq},  for $p\neq q$ it may happen that 
\begin{align*}
\max\left\{\frac{p+2+q^{-1}}{pq-1},\frac{q+2+p^{-1}}{pq-1}\right\}\geqslant\frac{n-1}{2}
\end{align*} (that is, $(p,q)$ belongs to the blow-up region in the $p$ - $q$ plane) even though one among $p,q$ is greater than the Strauss exponent.

The goal of this paper is to prove for the weakly coupled system \eqref{weakly coupled system} blow-up results for the same range of pair $(p,q)$ as in the corresponding results for  \eqref{weakly coupled system wave eq} and, furthermore, the same upper bound estimates for the lifespan of local solutions.

 From a more technical point of view, in this paper we will generalize the approaches for the Cauchy problem \eqref{sem wave eq scatt} in the case of a power nonlinearity developed by \cite{LT18Scatt} in the subcritical case and \cite{WakYor18damp} in the critical case to the study of a weakly coupled system of semilinear weave equations with damping terms in the scattering case. In the subcritical case the multiplier introduced in \cite{LT18Scatt} plays a fundamental role, in order to make the iteration frame for our model analogous to the one for the corresponding case without damping.  In the critical case, however, a nontrivial generalization of the approach by Wakasa-Yordanov is necessary, in order to take into account of the asymmetric behavior of the model on the critical curve except for the cusp point $p=q$. This situation will be dealt with the aid of an asymmetric frame in the iteration scheme. On the other hand, in the special case $p=q$ the situation is completely symmetric to what happens in the case of a single equation. 
 
 Finally, let us point out that, due to the general structure of the coefficients for the damping terms, we may not apply the revisited test function method recently developed by Ikeda-Sobajima-Wakasa for the classical wave equation in \cite{ISW18}, whose approach is based on a family of self-similar solutions (see also \cite{IS17} for the application of this method to the semilinear heat, damped wave and  Schr\"odinger equations  and \cite{IS18,PT18,Pal19} in the scale-invariant case).

Before stating the main results of this paper, let us introduce a suitable notion of energy solutions according to \cite{LTW17}. 

%

\begin{defn} \label{def energ sol intro} Let $u_0,v_0\in H^1(\mathbb{R}^n)$ and $u_1,v_1\in L^2(\mathbb{R}^n)$.
We say that $(u,v)$ is an energy solution of \eqref{weakly coupled system} on $[0,T)$ if
\begin{align*}
& u\in \mathcal{C}([0,T),H^1(\mathbb{R}^n))\cap \mathcal{C}^1([0,T),L^2(\mathbb{R}^n))\cap L^q_{\loc}([0,T)\times\mathbb{R}^n), \\
& v\in \mathcal{C}([0,T),H^1(\mathbb{R}^n))\cap \mathcal{C}^1([0,T),L^2(\mathbb{R}^n))\cap L^p_{\loc}([0,T)\times\mathbb{R}^n)
\end{align*}
satisfy $u(0,x)=\varepsilon u_0(x),v(0,x)=\varepsilon v_0(x)$ in $H^1(\mathbb{R}^n)$,
\begin{align} 
\int_{\mathbb{R}^n} & \partial_t u(t,x)\phi(t,x)\,dx-\int_{\mathbb{R}^n}\varepsilon u_1(x)\phi(0,x)\,dx - \int_0^t\int_{\mathbb{R}^n} \partial_t u(s,x)\phi_s(s,x) \,dx\, ds\notag \\
& \ +\int_0^t\int_{\mathbb{R}^n}\nabla u(s,x)\cdot\nabla\phi(s,x)\, dx\, ds+\int_0^t\int_{\mathbb{R}^n}b_1(s)\partial_t u(s,x) \phi(s,x)\,dx\, ds  \notag \\
& =\int_0^t \int_{\mathbb{R}^n}|v(s,x)|^p\phi(s,x)\,dx \, ds  \label{def u}
\end{align} and 
\begin{align} 
\int_{\mathbb{R}^n} & \partial_t v(t,x)\psi(t,x)\,dx-\int_{\mathbb{R}^n}\varepsilon v_1(x)\psi(0,x)\,dx - \int_0^t\int_{\mathbb{R}^n} \partial_t v(s,x)\psi_s(s,x) \,dx\, ds\notag \\
& \ +\int_0^t\int_{\mathbb{R}^n}\nabla v(s,x)\cdot\nabla\psi(s,x)\, dx\, ds+\int_0^t\int_{\mathbb{R}^n} b_2(s) \partial_t v(s,x)\psi(s,x)\,dx\, ds  \notag \\
& =\int_0^t \int_{\mathbb{R}^n}|u(s,x)|^q\psi(s,x)\,dx \, ds  \label{def v}
\end{align}
for any $\phi,\psi \in \mathcal{C}_0^\infty([0,T)\times\mathbb{R}^n)$ and any $t\in [0,T)$.
\end{defn}
 After a further step of integrations by parts, requiring further that the functions $b_1,b_2$ are continuously differentiable,  \eqref{def u} and \eqref{def v} provide
 \begin{align} 
& \int_{\mathbb{R}^n}  \big( \partial_t u(t,x)\phi(t,x)- u(t,x)\phi_s(t,x)+ b_1(t) u(t,x)\phi(t,x)\big)\,dx\notag \\ & \quad -\int_{\mathbb{R}^n}  \big( \varepsilon u_1(x)\phi(0,x)- \varepsilon u_0(x)\phi_s(0,x)+ b_1(0) \varepsilon u_0(x)\phi(0,x)\big)\,dx \notag\\ & \quad +  \int_0^t\int_{\mathbb{R}^n} u(s,x)\big(\phi_{ss}(s,x)-\Delta \phi(s,x)-\partial_s(b_1(s)\phi(s,x)\big) \,dx\, ds\notag \\
& \quad =\int_0^t \int_{\mathbb{R}^n}|v(s,x)|^p\phi(s,x)\,dx \, ds  \label{def u weak}
\end{align} and 
\begin{align} 
& \int_{\mathbb{R}^n}  \big( \partial_t v(t,x)\psi(t,x)- v(t,x)\psi_s(t,x)+ b_2(t) v(t,x)\psi(t,x)\big)\,dx\notag \\ & \quad -\int_{\mathbb{R}^n}  \big( \varepsilon v_1(x)\psi(0,x)- \varepsilon v_0(x)\psi_s(0,x)+ b_2(0) \varepsilon v_0(x)\psi(0,x)\big)\,dx \notag\\ & \quad +  \int_0^t\int_{\mathbb{R}^n} v(s,x)\big(\psi_{ss}(s,x)-\Delta \psi(s,x)-\partial_s(b_2(s)\psi(s,x)\big) \,dx\, ds\notag \\
& \quad =\int_0^t \int_{\mathbb{R}^n}|u(s,x)|^q\psi(s,x)\,dx \, ds.  \label{def v weak}
\end{align}
 In particular, letting $t\rightarrow T$, we find that $(u,v)$ fulfills the definition of weak solution to \eqref{weakly coupled system}.
 
Let us state the blow-up result for \eqref{weakly coupled system} in the subcritical case.

\begin{thm}\label{Thm blowup iteration} Let $b_1,b_2\in \mathcal{C}([0,\infty))\cap L^1([0,\infty))$ be nonnegative functions. Let us consider $p,q>1$ satisfying 
\begin{align}\label{critical exponent wave like case system}
\max\left\{\frac{p+2+q^{-1}}{pq-1},\frac{q+2+p^{-1}}{pq-1}\right\}>\frac{n-1}{2}\,.
\end{align} 

Assume that $u_0,v_0\in H^1(\mathbb{R}^n)$ and $u_1,v_1\in  L^2(\mathbb{R}^n)$ are nonnegative, pairwise nontrivial and compactly supported in $B_R\doteq \{x\in \mathbb{R}^n: |x|\leqslant R\}$ functions. 


Let $(u,v)$ be an energy solution of \eqref{weakly coupled system} with lifespan $T=T(\varepsilon)$ such that
\begin{align} \label{support condition solution}
\supp u, \, \supp v \subset\{(t,x)\in  [0,T)\times\mathbb{R}^n: |x|\leqslant t+R \}.
\end{align} Then, there exists a positive constant $\varepsilon_0=\varepsilon_0(u_0,u_1,v_0,v_1,n,p,q,b_1,b_2,R)$ such that for any $\varepsilon\in (0,\varepsilon_0]$ the solution $(u,v)$ blows up in finite time. Moreover,
 the upper bound estimate for the lifespan
\begin{align} \label{lifespan upper bound estimate}
T(\varepsilon)\leqslant C \varepsilon ^{-\max\{F(n,p,q),F(n,q,p)\}^{-1}}
\end{align} holds, where C is an independent of $\varepsilon$, positive constant and 
\begin{align}\label{def F(n,p,q) function}
F(n,p,q)\doteq \frac{p+2+q^{-1}}{pq-1}-\frac{n-1}{2}.
\end{align}
\end{thm}



\begin{cor} \label{Corollary A} Let $n=1$ and $p,q>1$, or $n=2$ and $1<p,q<2$. Furthermore, we assume that $(p,q)$ and $(u_0,u_1,v_0,v_1)$ satisfy the same assumptions as in Theorem \ref{Thm blowup iteration}. If $$\int_{\mathbb{R}^n}u_1(x) dx \neq 0 \ \ \mbox{and} \ \ \int_{\mathbb{R}^n}v_1(x) dx \neq 0,$$ then, the lifespan estimate \eqref{lifespan upper bound estimate} can be improved as follow
\begin{align*}
T(\varepsilon)\leqslant C \varepsilon^{-\max\{G(n,p,q),G(n,q,p)\}^{-1}},
\end{align*} where
\begin{align} \label{def G(n,p,q) function}
G(n,p,q) \doteq \frac{2(1+p^{-1})}{pq-1}-\frac{n}{p}+n-2.
\end{align}
\end{cor}

\begin{cor} \label{Corollary B} Let $n=2$ and $1<p<2$, $q\geqslant 2$. Furthermore, we assume that $(p,q)$ and $(u_0,u_1,v_0,v_1)$ satisfy the same assumptions as in Theorem \ref{Thm blowup iteration}. If $$\int_{\mathbb{R}^2}u_1(x) dx \neq 0,$$ then, the lifespan estimate \eqref{lifespan upper bound estimate} can be improved as follow
\begin{align*}
T(\varepsilon)\leqslant C \varepsilon^{-\max\{F(n,p,q),G(n,p,q)\}^{-1}},
\end{align*} where $F(n,p,q)$ and $G(n,p,q)$ are defined by \eqref{def F(n,p,q) function} and \eqref{def G(n,p,q) function}, respectively.
\end{cor}

\begin{cor} \label{Corollary C} Let $n=2$ and $1<q<2$, $p\geqslant 2$. Furthermore, we assume that $(p,q)$ and $(u_0,u_1,v_0,v_1)$ satisfy the same assumptions as in Theorem \ref{Thm blowup iteration}. If $$\int_{\mathbb{R}^2}v_1(x) dx \neq 0,$$ then, the lifespan estimate \eqref{lifespan upper bound estimate} can be improved as follow
\begin{align*}
T(\varepsilon)\leqslant C \varepsilon^{-\max\{F(n,q,p),G(n,q,p)\}^{-1}},
\end{align*} where $F(n,p,q)$ and $G(n,p,q)$ are defined by \eqref{def F(n,p,q) function} and \eqref{def G(n,p,q) function}, respectively.
\end{cor}


In the critical case we have the following result. 


\begin{thm}\label{Thm critical case}
Let $b_1,b_2\in \mathcal{C}^1([0,\infty))\cap L^1([0,\infty))$ be nonnegative functions and let $n\geqslant 2$.   Let us consider $p,q>1$ satisfying
\begin{align}\label{critical exponent wave like case system, critical case}
\max\left\{\frac{p+2+q^{-1}}{pq-1},\frac{q+2+p^{-1}}{pq-1}\right\}=\frac{n-1}{2}. 
\end{align} 
Assume that $u_0,v_0\in H^1(\mathbb{R}^n)$ and $u_1,v_1\in  L^2(\mathbb{R}^n)$ are nonnegative, pairwise nontrivial and compactly supported in $B_R$. 

Let $(u,v)$ be an energy solution of \eqref{weakly coupled system} with lifespan $T=T(\varepsilon)$ that satisfies \eqref{support condition solution}. Then, there exists a positive constant $\varepsilon_0=\varepsilon_0(u_0,u_1,v_0,v_1,n,p,q,b_1,b_2,R)$ such that for any $\varepsilon\in (0,\varepsilon_0]$ the solution $(u,v)$ blows up in finite time. Moreover,
 the upper bound estimates for the lifespan
\begin{align} \label{lifespan upper bound estimate, critical case}
T(\varepsilon)\leqslant 
\begin{cases} 
\exp \big(C \varepsilon ^{-\min\{q(pq-1),p(pq-1)\}}\big) & \mbox{if} \ \ p\neq q, \\
\exp \big(C \varepsilon ^{-p(p-1)}\big) & \mbox{if} \ \ p=q
\end{cases}
\end{align} hold, where C is an independent of $\varepsilon$, positive constant and $F=F(n,p,q)$ is defined by \eqref{def F(n,p,q) function}.
\end{thm}

\begin{rem} The upper bound estimates \eqref{lifespan upper bound estimate} and \eqref{lifespan upper bound estimate, critical case} for the lifespan coincide with the sharp estimates for the lifespan of local solutions to the weakly coupled system of semilinear wave equations with power nonlinearities. However, as we do not deal with global in time existence results for \eqref{weakly coupled system} in the present work, we do not derive a lower bound estimate for $T(\varepsilon)$.
\end{rem}

\begin{rem} Let us point out explicitly that in the critical case we need to require more regularity for the time-dependent coefficients $b_1,b_2$ in comparison to the subcritical case. Namely, $b_1,b_2$ are assumed of class $\mathcal{C}^1$ rather than being merely continuous. The reason of this stronger assumption is that in the critical case we shall employ \eqref{def u weak}-\eqref{def v weak} in place of \eqref{def u}-\eqref{def v}, in order to find the coupled system of ordinary integral inequalities for suitable functionals, whose dynamic is studied to prove the blow-up result.
\end{rem}

In this paper we study the nonexistence of global in time solutions for a semilinear weakly coupled system of damped wave equations in the scattering producing case with power nonlinearities and the corresponding upper bound for the lifespan in the same range of powers $(p,q)$ as for the analogous system without damping terms. In two forthcoming papers \cite{PalTak19der,PalTak19mix} we will consider as well the case with nonlinearities of derivative type and of mixed type 
 for a semilinear weakly coupled system of damped wave equations in the scattering case.

The remaining part of this paper is organized as follows: in Section \ref{Section lower bounds} we recall a multiplier, that has been introduced in \cite{LT18Scatt} in order to study the corresponding single semilinear equation, and its properties and we derive some lower bounds for certain functionals related to a local solution; then, in Section \ref{Section proof main thm} we prove Theorem \ref{Thm blowup iteration} by using the preparatory results from Section \ref{Section lower bounds} and an iterative method. Finally, in Section \ref{Section critical case} we prove the result in the critical case adapting the approach from \cite{WakYor18,WakYor18damp} for a weakly coupled system. In particular, the slicing method is employed in order to deal with logarithmic factors in the iteration argument.

\subsection*{Notations} Throughout this paper we will use the following notations: $B_R$ denotes the ball around the origin with radius $R$; $f \lesssim g$ means that there exists a positive constant $C$ such that $f \leqslant Cg$ and, similarly, for $f\gtrsim g$; 
finally, as in the introduction, $p_0(n)$ denotes the Strauss exponent.

\section{Definition of the multipliers and lower bounds of the functionals}
\label{Section lower bounds}

The arguments used in this section are similar to some of those employed in \cite[Section 3]{LT18Scatt}. However, for the sake of self-containedness and readability of the paper, we will provide them.
 
\begin{defn} Let $b_1,b_2\in \mathcal{C}([0,\infty))\cap L^1([0,\infty))$ be the nonnegative, time-dependent coefficients in \eqref{weakly coupled system}. We define the corresponding \emph{multipliers}
\begin{align*}
m_j(t)\doteq \exp \bigg(-\int_t^\infty b_j(\tau) d\tau \bigg) \qquad \mbox{for} \ \  t\geqslant 0 \ \ \mbox{and} \ j=1,2. 
\end{align*}
\end{defn}

Since $b_1,b_2$ are nonnegative functions, it follows that $m_1,m_2$ are increasing functions. Moreover, due to the fact that these coefficients are summable, we get also that these multipliers are bounded and 
\begin{align}
m_j(0)\leqslant m_j(t) \leqslant 1 \qquad \mbox{for} \ \  t\geqslant 0 \ \ \mbox{and} \ j=1,2. \label{boundedness multipliers}
\end{align}
A fundamental property of these multipliers is the relation with the corresponding derivatives. More precisely,
\begin{align}
m_j'(t) = b_j(t) \, m(t)  \qquad \mbox{for} \ \   j=1,2. \label{deribative multiplier}
\end{align} Such a relation will play a fundamental role in the remaining part of this section, which  is devoted to the determination of lower bounds for the spatial integral of the nonlinear terms and to the deduction of a pair of coupled integral inequalities for the spatial averages of the components of a local solution to \eqref{weakly coupled system}.
\begin{lem} \label{lemma lower bound nonlinearities}
Let us assume that $u_0, u_1,  v_0,  v_1$ are nonnegative, pairwise nontrivial and compactly supported in $B_R$ for some $R>0$. Let $(u,v)$ be a local (in time) energy solution to \eqref{weakly coupled system} satisfying \eqref{support condition solution}. Then,
there exist two constants $C_1=C_1(u_0,u_1,b_1,q,R)>0$ and $K_1=K_1(v_0,v_1,b_2,p,R)>0$, independent of $\varepsilon$ and $t$, such that for any $t\geqslant 0$ and $p,q>1$, the following estimates hold:
\begin{align}\label{Priori u^q}
\int_{\mathbb{R}^n}|u(t,x)|^q dx & \geqslant C_1\varepsilon^q(1+t)^{n-1-\frac{n-1}2 q}, \\
\label{Priori v^p}
\int_{\mathbb{R}^n}|v(t,x)|^p dx & \geqslant K_1\varepsilon^p(1+t)^{n-1-\frac{n-1}2 p}.
\end{align}
\end{lem}

\begin{proof} 
Let us define the functionals
\begin{align*}
U_1(t)\doteq \int_{\mathbb{R}^n}u(t,x)\Psi(t,x)\, dx  \qquad  \mbox{and} \qquad V_1(t)\doteq \int_{\mathbb{R}^n}v(t,x)\Psi(t,x)\, dx 
\end{align*}
where $\Psi=\Psi(t,x)\doteq e^{-t} \Phi(x)$ and 
\begin{align}
\Phi=\Phi(x)\doteq \begin{cases} e^{x}+e^{-x} & \mbox{for} \ \ n=1, \\\displaystyle{\int_{\mathbb{S}^{n-1}} \, e^{\omega \cdot x}\, dS_{\omega}} & \mbox{for} \ \ n\geqslant 2\end{cases} \label{def eigenfunction laplace op}
\end{align} is an eigenfunction of the Laplace operator, as $\Delta \Phi =\Phi$. Then, by H\"{o}lder inequality, we have
\begin{align}\int_{\mathbb{R}^n}|u(t,x)|^qdx\geqslant |U_1(t)|^q\bigg(\int_{|x|\leqslant t+R}\Psi^{q'}(t,x)dx\bigg)^{-(q-1)}\ ,\label{factor}\\
\int_{\mathbb{R}^n}|v(t,x)|^pdx\geqslant |V_1(t)|^p\bigg(\int_{|x|\leqslant t+R}\Psi^{p'}(t,x)dx\bigg)^{-(p-1)}\ , \notag 
\end{align} where $p',q'$ denote the conjugate exponents of $p,q$, respectively. We will prove now \eqref{Priori u^q} by using \eqref{factor}, the proof of \eqref{Priori v^p} being analogous.
The next steps consist in determining a lower bound for $U_1(t)$ and an upper bound for the integral $\int_{|x|\leqslant t+R}\Psi^{q'}(t,x)dx$,  respectively.

Due to the support property for $u$, we can apply the definition of energy solution with test functions that are not compactly supported. Applying the definition of energy solution with $\Psi$ as test function and differentiating with respect to $t$ the obtained relation, we find for any $t\in (0,T)$ 
\begin{align*}
 & \frac{d}{dt}  \int_{\mathbb{R}^n} u_t(t,x)\Psi(t,x) \, dx + \int_{\mathbb{R}^n}\Big(-u_t(t,x)\Psi_t(t,x)+\nabla u(t,x)\cdot  \nabla \Psi(t,x)+b_1(t)u_t(t,x) \Psi(t,x) \Big)dx  	\\ & \quad = \int_{\mathbb{R}^n}|v(t,x)|^p\Psi(t,x) \,dx. 
\end{align*}  
Rearranging the previous relation, we get 
\begin{align*}
\int_{\mathbb{R}^n}|v(t,x)|^p\Psi(t,x) dx & =  \frac{d}{dt}  \int_{\mathbb{R}^n} u_t(t,x)\Psi(t,x) \, dx - \int_{\mathbb{R}^n}u_t(t,x)\Psi_t(t,x)\, dx \\ & \quad - \int_{\mathbb{R}^n} u(t,x)\, \Delta \Psi(t,x)\, dx + b_1(t) \int_{\mathbb{R}^n} u_t(t,x) \Psi(t,x) \,dx \\
& =  \frac{d}{dt}  \int_{\mathbb{R}^n} u_t(t,x)\Psi(t,x) \, dx + b_1(t) \int_{\mathbb{R}^n} u_t(t,x) \Psi(t,x) \,dx  \\ & \quad + \int_{\mathbb{R}^n}\big(u_t(t,x) \Psi(t,x)  -u(t,x)\Psi(t,x)\big) dx,
\end{align*}
where in the last step we used the properties $\Psi_t=-\Psi$ and $\Delta \Psi= \Psi$.
 Multiplying both sides of the previous relation by the multiplier $m_1$ and employing \eqref{deribative multiplier}, we obtain 
 \begin{align*}
m_1(t) \int_{\mathbb{R}^n} |v(t,x)|^p\Psi(t,x)\,  dx & =  \frac{d}{dt} \bigg(m_1(t) \int_{\mathbb{R}^n} u_t(t,x)\Psi(t,x) \, dx \bigg) \\ & \quad + m_1(t)\int_{\mathbb{R}^n}\big(u_t(t,x) \Psi(t,x)  -u(t,x)\Psi(t,x)\big) dx.
 \end{align*} 
Integrating the last equality over $[0,t]$, we find 
\begin{align*}
\int_0^t m_1(s) \int_{\mathbb{R}^n} |v(s,x)|^p\Psi(s,x) \, dx\,  ds & =  m_1(t) \int_{\mathbb{R}^n} u_t(t,x)\Psi(t,x) \, dx -\varepsilon \, m_1(0) \int_{\mathbb{R}^n} u_1(x)\, \Phi(x) \, dx \\ & \quad + \int_0^t m_1(s)\int_{\mathbb{R}^n}\big(u_s(s,x) \Psi(s,x)  -u(s,x)\Psi(s,x)\big) dx \, ds.
\end{align*}
Noticing that
\begin{align*}
 \int_0^t   m_1(s) \int_{\mathbb{R}^n} u_s(s,x) & \Psi(s,x) \,  dx \, ds  \\ & = 
m_1(t)\int_{\mathbb{R}^n} u(t,x) \,\Psi(t,x) \,  dx - \varepsilon \, 
m_1(0)\int_{\mathbb{R}^n} u_0(x) \,\Phi(x) \,  dx \\ & \quad - \int_0^t \int_{\mathbb{R}^n} u(s,x)  \Big(m_1'(s) \Psi(s,x)+m_1(s) \Psi_s(s,x)\Big)  dx \, ds \\
& = 
m_1(t)\int_{\mathbb{R}^n} u(t,x) \Psi(t,x) \,  dx - \varepsilon \, 
m_1(0)\int_{\mathbb{R}^n} u_0(x)\, \Phi(x) \,  dx \\  & \quad - \int_0^t \int_{\mathbb{R}^n} u(s,x) \, b_1(s)\,  m_1(s) \Psi(s,x) \,  dx \, ds +\int_0^t  m_1(s) \int_{\mathbb{R}^n} u(s,x)   \Psi(s,x) \,  dx \, ds,
\end{align*} it follows
\begin{align*}
\int_0^t & m_1(s) \int_{\mathbb{R}^n} |v(s,x)|^p\Psi(s,x) \, dx\,  ds  + \int_0^t b_1(s)\,  m_1(s) \int_{\mathbb{R}^n} u(s,x) \,  \Psi(s,x) \,  dx \, ds  \\ & \quad \quad + \varepsilon \, m_1(0) \int_{\mathbb{R}^n} \big(u_0(x)+u_1(x)\big) \Phi(x) \, dx   \\ & \quad  =  m_1(t) \int_{\mathbb{R}^n} \big( u_t(t,x)\Psi(t,x) +u(t,x) \Psi(t,x)\big) \, dx  
\\ & \quad  =  m_1(t) \, \frac{d}{dt} \int_{\mathbb{R}^n}  u(t,x)\Psi(t,x) \, dx  + 2m_1(t) \int_{\mathbb{R}^n} u(t,x) \Psi(t,x) \, dx.  
\end{align*}
Using the definition of the functional $U_1$, from the previous relation we derive the inequality
\begin{align*}
m_1(t) \big(U'_1(t)+2 U_1(t)\big) \geqslant \varepsilon \, m_1(0) \, C(u_0,u_1) + \int_0^t b_1(s)\,  m_1(s) \, U_1(s) \, ds ,
\end{align*} where $C(u_0,u_1)\doteq  \int_{\mathbb{R}^n} \big(u_0(x)+u_1(x)\big) \Phi(x) \, dx $. Using the boundedness of the multiplier $m_1$, we get
\begin{align}
U'_1(t)+2 U_1(t)  &\geqslant \varepsilon \, \frac{m_1(0)}{m_1(t)} \, C(u_0,u_1) +\frac{1}{m(t)} \int_0^t b_1(s)\,  m_1(s)\, U_1(s) \, ds \notag\\
&\geqslant \varepsilon \, m_1(0) \, C(u_0,u_1) +\frac{1}{m(t)} \int_0^t b_1(s)\,  m_1(s)\,  U_1(s) \, ds. \label{diff inequality U1}
\end{align} A multiplication of both sides in the last estimate by $e^{2t}$ and an integration over $[0,t]$ yield
\begin{align}
e^{2t}U_1(t)\geqslant U_1(0)+ \varepsilon \, \frac{m_1(0)}{2} \, C(u_0,u_1) (e^{2t}-1)+\int_0^t  \frac{e^{2s}}{m(s)} \int_0^\tau b_1(\tau)\,  m_1(\tau) \, U_1(\tau) \, d\tau  \, ds. \label{comparison}
\end{align}
A comparison argument proves the positiveness of the functional $U_1$. 
Due to the fact that initial data are pairwise nontrivial, at least one among $u_0,u_1$ is not identically $0$.
In the first case $u_0\not \equiv 0$, since $u_0\geqslant 0$ implies $U_1(0)>0$, by continuity it holds $U_1(t)>0$ at least in a right neighborhood of $t=0$. If $t_0>0$ was the smallest value such that $U_1(t_0)=0$, then, evaluation of \eqref{comparison} in $t=t_0$ would provide a contradiction. In the second case $u_0\equiv 0$ and $u_1\not \equiv 0$, we can employ \eqref{diff inequality U1} to get a contradiction. Indeed, in this case we have $U_1(0)=0$ and $U_1'(0)=\varepsilon \int_{\mathbb{R}^n} u_1(x)\Phi(x)\, dx >0$. By continuity, $U_1'(t)>0$ for any $t\in [0,t_1)$ with $t_1>0$. Therefore, $U_1$ is strictly increasing, and then positive, in $(0,t_1)$. Let us assume by contradiction that $t_2>t_1$ is the smallest value such that $U_1(t_2)=0$. Consequently, $U_1'(t_2)\leqslant 0$ (if $U_1'(t_2)$ was positive, then, $U_1$ would be strictly increasing in a neighborhood of $t_2$, but this would contradict the definition of $t_2$, since there would be a smaller zero, $U_1$ being negative in a left neighborhood of $t_2$). If we plug $U_1(t_2)=0, U_1'(t_2)\leqslant 0$ and $U_1(t)>0$ for $t\in (0,t_2)$ in \eqref{diff inequality U1}, we find the contradiction we were looking for.

 In particular, due to the fact that $U_1$ is positive, \eqref{comparison} implies 
\begin{align}
U_1(t)\geqslant e^{-2t} U_1(0)+ \varepsilon \, \frac{m_1(0)}{2} \, C(u_0,u_1) (1-e^{-2t}) \gtrsim \varepsilon. \label{num}
\end{align}

 The integral involving $\Psi^{q'}$ in the right-hand side of \eqref{factor} can be estimated in a standard way (cf. estimate (2.5) in \cite{YZ06}), namely,
\begin{align}
\int_{|x|\leqslant t+R}\Psi^{q'}(t,x)\, dx & \leqslant e^{-\frac{q}{q-1} t}\int_{|x|\leqslant t+R}\Phi^{q'}(x)\, dx \leqslant C_{\Phi,R}\, (1+t)^{n-1-\frac{n-1}{2}\frac{q}{q-1}},\label{deno}
\end{align}
where $C_{\Phi,R}$ is a suitable positive constant.
Combing the estimate \eqref{num}, \eqref{deno} and \eqref{factor}, we find \eqref{Priori u^q}. This concludes the proof.
\end{proof}

\begin{rem} As we have already mentioned the proof of Lemma \ref{lemma lower bound nonlinearities} follows the approach from Section 3 in \cite{LT18Scatt}. However, the same estimates can be proved by following the proof of Lemma 5.1 in \cite{WakYor18}, by working with a different functional in place of $U_1$.
\end{rem}

\section{Subcritical case: Proof of Theorem \ref{Thm blowup iteration}} \label{Section proof main thm}

Let us consider a local solution $(u,v)$ of \eqref{weakly coupled system} on $[0,T)$ and define the following couple of time-dependent functionals related to this solution:
\begin{align*}
U(t)   \doteq \int_{\mathbb{R}^n} u(t,x) \, dx ,  \quad
V(t)   \doteq \int_{\mathbb{R}^n} v(t,x) \, dx.
\end{align*}
The proof of Theorem \ref{Thm blowup iteration} consists of two parts. In the first part we determine a pair of coupled integral inequalities for $U$ and $V$, while in the second one an iteration argument is used so that the blow-up of $(U,V)$ in finite time can be shown. 

\subsection{Determination of the iteration frame} 

If we choose  $\phi=\phi(s,x)$ and $\psi=\psi(s,x)$ in \eqref{def u} and in \eqref{def v}, respectively, satisfying  $\phi\equiv 1 \equiv \psi$ on $\{(x,s)\in  [0,t]\times \mathbb{R}^n :|x|\leqslant s+R\}$, then, we find
\begin{align*}&\int_{\mathbb{R}^n}\partial_t u(t,x)\,dx-\int_{\mathbb{R}^n}\partial_t u(0,x)\,dx+\int_0^t \int_{\mathbb{R}^n}b_1(s) \, \partial_t u(s,x) \, dx \, ds
=\int_0^t \int_{\mathbb{R}^n}|v(s,x)|^pdx \, ds, \\
&\int_{\mathbb{R}^n}\partial_t v(t,x)\,dx-\int_{\mathbb{R}^n}\partial_t v(0,x)\,dx+\int_0^t \int_{\mathbb{R}^n} b_2(s) \, \partial_t v(s,x) \, dx \, ds
=\int_0^t \int_{\mathbb{R}^n}|u(s,x)|^q dx \, ds
\end{align*}
or, equivalently,
\begin{align*}
& U'(t)-U'(0)+\int_0^t b_1(s)  U'(s)\,ds =\int_0^t \int_{\mathbb{R}^n}|v(s,x)|^pdx \, ds, \\
& V'(t)-V'(0)+\int_0^t b_2(s)  V'(s)\,ds=\int_0^t \int_{\mathbb{R}^n}|u(s,x)|^q dx \, ds.
\end{align*}
Differentiating with repect to $t$ the previous equalities, we arrive at
\begin{align}\label{U-Dyn}
U''(t)+b_1(t) U'(t)=\int_{\mathbb{R}^n}|v(t,x)|^p dx, \\
\label{V-Dyn}
V''(t)+b_2(t) V'(t)=\int_{\mathbb{R}^n}|u(t,x)|^q dx.
\end{align}
Multiplying \eqref{U-Dyn} by $m_1(t)$, we get
\begin{align*}
m_1(t)U''(t)+m_1(t) b_1(t) U'(t)= \frac{d}{dt} \big(m_1(t)U'(t)\big)= m_1(t) \int_{\mathbb{R}^n}|v(t,x)|^p dx.
\end{align*}
 Hence, integrating over $[0,t]$ and using the assumption $u_1\geqslant 0$, we obtain
\begin{align*}
m_1(t)U'(t)= m_1(0)U'(0) + \int_0^t m_1(s) \int_{\mathbb{R}^n}|v(s,x)|^p dx \, ds \geqslant  \int_0^t m_1(s) \int_{\mathbb{R}^n}|v(s,x)|^p dx \, ds.
\end{align*}
Consequently, using the boundedness of the multiplier $m_1$, from \eqref{boundedness multipliers} we have
\begin{align*}
U'(t)  \geqslant  \int_0^t \frac{m_1(s)}{m_1(t)} \int_{\mathbb{R}^n}|v(s,x)|^p dx \, ds \geqslant  m_1(0) \int_0^t \int_{\mathbb{R}^n}|v(s,x)|^p dx \, ds.
\end{align*} 
Since $u_0$ is nonnegative a further integration on $[0,t]$ provides
\begin{align}
 U(t)& \geqslant  m_1(0) \int_0^t \int_0^\tau \int_{\mathbb{R}^n}|v(s,x)|^p dx \, ds\, d\tau.  \label{iter1}
\end{align}
Moreover, due to H\"{o}lder inequality and the compactness of the support of solution with respect to $x$, from \eqref{iter1} we derive
\begin{align}
  U(t) &\geqslant C_0\int_0^t \int_0^\tau (1+s)^{-n(p-1)}|V(s)|^p ds\, d\tau, \label{iter2}
\end{align}
where $C_0\doteq m_1(0) (\meas(B_1))^{1-p}R^{-n(p-1)}>0.$

In a similar way, using the assumptions $v_0,v_1\geqslant 0$ and the properties of the multiplier $m_2$, from \eqref{V-Dyn} we may derive
\begin{align}
V(t)&\geqslant m_2(0) \int_0^t \int_0^\tau \int_{\mathbb{R}^n}|u(s,x)|^q dx \, ds\, d\tau   \label{iter3} \\
 &\geqslant K_0 \int_0^t \int_0^\tau (1+s)^{-n(q-1)}|U(s)|^q ds\, d\tau \label{iter4},
\end{align} where $K_0\doteq m_2(0)(\meas(B_1))^{1-q}R^{-n(q-1)}>0.$

\subsection{Iteration argument} \label{Subsection iteration argument subcritical case}
Now we can proceed with the second part of the proof, where we use a standard iteration argument (see for example \cite{LT18Scatt,TL1709} in the cae of a single equation or \cite{AKT00,Pal19} in the case of a weakly coupled system). We will apply an iteration method based on lower bound estimates \eqref{Priori u^q}, \eqref{Priori v^p}, \eqref{iter1}, \eqref{iter3} and  the iteration frame \eqref{iter2}, \eqref{iter4}.  

By using an induction argument, we prove that 
\begin{align}
U(t)& \geqslant D_j(1+t)^{-a_j}t^{b_j} \quad \mbox{for} \ \ t\geqslant 0, \label{lower bound U j} \\
V(t)& \geqslant \Delta_j(1+t)^{-\alpha_j}t^{\beta_j} \quad \mbox{for} \ \ t\geqslant 0, \label{lower bound V j}
\end{align} where $\{a_j\}_{j\geqslant 1}$, $\{b_j\}_{j\geqslant 1}$, $\{D_j\}_{j\geqslant 1}$, $\{\alpha_j\}_{j\geqslant 1}$, $\{\beta_j\}_{j\geqslant 1}$ and $\{\Delta_j\}_{j\geqslant 1}$ are suitable sequences of nonnegative real numbers to be determined afterwards.

We prove first the base case $j=1$. Plugging the lower bound estimate for the nonlinear term $|v|^p$ given by \eqref{Priori v^p} in \eqref{iter1}, we obtain for $t\geqslant 0$
\begin{align*}
U(t) & \geqslant m_1(0) K_1 \varepsilon^p  \int_{0}^t \int_{0}^\tau  (1+s)^{n-1-(n-1)\frac{p}{2}} \, ds\, d\tau \\
&  \geqslant  m_1(0) K_1\varepsilon^p (1+t)^{-(n-1)\frac{p}{2}} \int_{0}^t \int_{0}^\tau  s^{n-1} \, ds\, d\tau  \geqslant  \frac{m_1(0) K_1}{n(n+1)}\varepsilon^p (1+t)^{-(n-1)\frac{p}{2}}   t^{n+1} ,
\end{align*} which is the desired estimate, provided that we define
\begin{align*}
D_1 &\doteq \frac{m_1(0) K_1}{n(n+1)}\varepsilon^p ,\qquad  a_1 \doteq (n-1)\frac{p}{2}, \quad b_1\doteq n+1.
\end{align*}
Analogously, we can prove \eqref{lower bound V j} for $j=1$ combining \eqref{iter3} and \eqref{Priori u^q}, provided that
\begin{align*}
\Delta_1 & \doteq \frac{m_2(0) C_1}{n(n+1)}\varepsilon^q, \qquad \alpha_1 \doteq (n-1)\frac{q}{2}, \quad \beta_1\doteq n+1.
\end{align*}

Let us proceed with the inductive step: \eqref{lower bound U j} and \eqref{lower bound V j} are assumed to be true for $j\geqslant 1$, hence, we prove them for $j+1$. Let us combine \eqref{lower bound V j} in \eqref{iter2}. Then, since $\alpha_j$ and $\beta_j$ are positive numbers, we obtain
\begin{align*}
 U(t) &\geqslant C_0 \, \Delta_j^p  \int_{0}^t \int_{0}^\tau  (1+s)^{-n(p-1)-\alpha_j p}s^{\beta_j p} \, ds\, d\tau \geqslant C_0 \, \Delta_j^p (1+t)^{-n(p-1)-\alpha_j p} \int_{0}^t \int_{0}^\tau  s^{\beta_j p} \, ds\, d\tau \\
 & =  \frac{C_0 \, \Delta_j^p}{(\beta_j p+1)(\beta_j p+2)} (1+t)^{-n(p-1)-\alpha_j p}  t^{\beta_j p+2} , 
\end{align*} that is, \eqref{lower bound U j} for $j+1$ provided that
\begin{align*}
D_{j+1}\doteq \frac{C_0 \, \Delta_j^p}{(\beta_j p+1)(\beta_j p+2)}, \qquad a_{j+1}\doteq n(p-1)+\alpha_j p, \quad b_{j+1}\doteq \beta_j p +2. 
\end{align*}
Similarly, we can prove \eqref{lower bound V j} for $j+1$ combining \eqref{iter4} and \eqref{lower bound U j}, in the case in which
\begin{align*}
\Delta_{j+1}\doteq \frac{K_0 \, D_j^q }{(b_j q+1)(b_j q+2)} \qquad \alpha_{j+1}\doteq n(q-1)+a_j q, \quad \beta_{j+1}\doteq b_j q+2. 
\end{align*}
So, we proved the inductive step. In particular, the positiveness of the exponents $a_j,b_j,\alpha_j,\beta_j$ follows immediately by the recursive relations we required throughout the inductive step and by the fact that the initial terms $a_1,b_1,\alpha_1,\beta_1$ are nonnegative.

Let us determine now explicitly the representations for $a_j,b_j,\alpha_j,\beta_j$. Let us begin with the case in which $j$ is an odd integer. We start with $a_j$. Using the previous definitions and applying iteratively the obtained relation, we have
\begin{align}
a_j &= n(p-1)+\alpha_{j-1} p =  n(p-1)+\big(n(q-1)+a_{j-2} q\big) p = n(pq-1) + a_{j-2} pq \notag \\ &= n(pq-1) \sum_{k=0}^{(j-3)/2}(pq)^k + a_{1} (pq)^{\frac{j-1}{2}} = (n+a_{1}) (pq)^{\frac{j-1}{2}} -n \label{explicit epression a j dep a1} \\ & \hphantom{= n(pq-1) \sum_{k=0}^{(j-3)/2}(pq)^k + a_{1} (pq)^{\frac{j-1}{2}} } \ = \big(n+\tfrac{n-1}{2}p\big) (pq)^{\frac{j-1}{2}} -n. \label{explicit epression a j}
\end{align}
In a completely analogous way, for odd $j$ we get
\begin{align}
\alpha_j& = (n+\alpha_{1}) (pq)^{\frac{j-1}{2}} -n \label{explicit epression alpha j dep alpha1}\\ &= \big(n+\tfrac{n-1}{2}q\big) (pq)^{\frac{j-1}{2}} -n.  \label{explicit epression alpha j}
\end{align} 
For the sake of brevity, we do not derive the representations of $a_j$ and $\alpha_j$ for even $j$, as it is unnecessary for the proof of the theorem. 

Similarly, combining the definitions of $b_j$ and $\beta_j$,  for odd $j$ we have
\begin{align*}
b_j & =\beta_{j-1}p +2 = \big( b_{j-2} q +2\big) p +2 =  b_{j-2} pq+ 2(p+1) ,\\ 
\beta_j & = b_{j-1}q +2 = \big( \beta_{j-2} p +2\big) q +2 =  \beta_{j-2} pq+ 2(q+1).  
\end{align*} Also,
\begin{align}
b_j  = b_1(pq)^{\frac{j-1}{2}}+2(p+1) \sum_{k=0}^{(j-3)/2}(pq)^k & = \Big(b_1+\tfrac{2(p+1)}{pq-1}\Big)(pq)^{\frac{j-1}{2}} -\tfrac{2(p+1)}{pq-1} \label{explicit epression b j odd dep b1}\\ &= \Big(n+1+\tfrac{2(p+1)}{pq-1}\Big)(pq)^{\frac{j-1}{2}} -\tfrac{2(p+1)}{pq-1},  \label{explicit epression b j odd}\\
\beta_j  = \beta_1(pq)^{\frac{j-1}{2}}+2(q+1) \sum_{k=0}^{(j-3)/2}(pq)^k &= \Big(\beta_1+\tfrac{2(q+1)}{pq-1}\Big)(pq)^{\frac{j-1}{2}} -\tfrac{2(q+1)}{pq-1} \label{explicit epression beta j odd dep beta1}\\ &= \Big(n+1+\tfrac{2(q+1)}{pq-1}\Big)(pq)^{\frac{j-1}{2}} -\tfrac{2(q+1)}{pq-1}. \label{explicit epression beta j odd}
\end{align} In the case in which $j$ is even, from \eqref{explicit epression beta j odd} and \eqref{explicit epression b j odd} we have, respectively,
\begin{align}
b_j &= \beta_{j-1}p +2 = p \Big(n+1+\tfrac{2(q+1)}{pq-1}\Big)(pq)^{\frac{j-2}{2}} -\tfrac{2p(q+1)}{pq-1}+2, \label{explicit epression b j even} \\
\beta_j &= b_{j-1}q +2 = q \Big(n+1+\tfrac{2(p+1)}{pq-1}\Big)(pq)^{\frac{j-2}{2}} -\tfrac{2q(p+1)}{pq-1}+2. \label{explicit epression beta j even}
\end{align}
Thus, from \eqref{explicit epression b j odd}, \eqref{explicit epression beta j odd}, \eqref{explicit epression b j even} and \eqref{explicit epression beta j even}, we see that for any $j\geqslant 1$ the following estimates hold:
\begin{equation}
\begin{split}
b_j &< B_0 (pq)^{\frac{j-1}{2}}, \qquad \beta_j < \widetilde{B}_0 (pq)^{\frac{j-1}{2}} \qquad \mbox{for} \ j \ \mbox{odd}, \\
b_j &< B_0 (pq)^{\frac{j}{2}}, \qquad \ \ \beta_j < \widetilde{B}_0 (pq)^{\frac{j}{2}} \ \ \ \qquad \mbox{for} \ j \ \mbox{even},
\end{split} \label{upper bound b j and beta j subcrit}
\end{equation} where $B_0=B_0(p,q,n)$ and $\widetilde{B}_0=\widetilde{B}_0(p,q,n)$ are positive and independent of $j$ constants. 
 
The next step is to derive lower bounds for $D_j$ and $\Delta_j$. From  the definition of $D_j$ and $\Delta_j$ it follows immediately 
\begin{align}\label{lower bound Dj Deltaj n1}
D_{j} & \geqslant \frac{C_0  }{b_{j}^2}  \Delta_{j-1}^p \ \ \mbox{and} \ \ \
\Delta_{j}  \geqslant \frac{K_0 }{\beta_{j}^2} D_{j-1}^q .
\end{align} 
Hence, due to \eqref{upper bound b j and beta j subcrit}, coupling the inequalities in \eqref{lower bound Dj Deltaj n1}, it follows
\begin{align} \label{lower bound Dj  n2}
D_j & \geqslant \frac{C_0}{B_0^2}\frac{\Delta_{j-1}^p}{(pq)^{j-1}} \geqslant \frac{C_0 K_0^p}{B_0^2} \frac{D_{j-2}^{pq}}{(pq)^{j-1}\beta_{j-1}^{2p}} \geqslant \frac{C_0 K_0^p}{B_0^{2}\widetilde{B}_0^{2p}} \frac{D_{j-2}^{pq}}{\big((pq)^{p+1}\big)^{j-1}}  = \frac{\widetilde{C} D_{j-2}^{pq}}{\big((pq)^{p+1}\big)^{j-1}},\\  \label{lower bound Deltaj  n2}
\Delta_j & \geqslant \frac{K_0}{\widetilde{B}_0^2}\frac{D_{j-1}^q}{(pq)^{j-1}} \geqslant \frac{K_0 C_0^q}{\widetilde{B}_0^2} \frac{\Delta_{j-2}^{pq}}{(pq)^{j-1}b_{j-1}^{2q}} \geqslant \frac{K_0 C_0^q}{\widetilde{B}_0^{2}B_0^{2q}} \frac{\Delta_{j-2}^{pq}}{\big((pq)^{q+1}\big)^{j-1}}  =  \frac{\widetilde{K} \Delta_{j-2}^{pq}}{\big((pq)^{q+1}\big)^{j-1}},
\end{align} where $\widetilde{C}\doteq C_0 K_0^p/B_0^{2}\widetilde{B}_0^{2p}$ and $\widetilde{K}\doteq K_0 C_0^q/\widetilde{B}_0^{2}B_0^{2q}$. 
By \eqref{lower bound Dj  n2}, if $j$ is odd, then, we have
\begin{align*}
\log D_j & \geqslant pq \log D_{j-2} -(j-1)(p+1)\log (pq) +\log \widetilde{C} \\
& \geqslant (pq)^2 \log D_{j-4} -\big((j-1)+(j-3)pq\big)(p+1)\log (pq) +\big(1+pq\big)\log \widetilde{C} \\
 & \geqslant \ \cdots  \geqslant (pq)^{\frac{j-1}{2}} \log D_{1} -\Bigg(\sum_{k=1}^{(j-1)/2}(j+1-2k)\,(pq)^{k-1}\Bigg)(p+1)\log (pq) +\Bigg(\sum_{k=0}^{(j-3)/2}(pq)^k\Bigg)\log \widetilde{C}.
\end{align*}
Using an inductive argument, the following formula can be shown:
\begin{align*}
\sum_{k=1}^{(j-1)/2}(j+1-2k)\, (pq)^{k-1}= \frac{1}{pq-1} \bigg(2(pq) \, \frac{(pq)^{\frac{j-1}{2}}-1}{pq-1}-j+1\bigg).
\end{align*}
Also,
\begin{align*}
\log D_j & \geqslant (pq)^{\frac{j-1}{2}} \bigg[\log D_{1}-\frac{2(pq)(p+1)}{(pq-1)^2}\log(pq)+\frac{\log \widetilde{C}}{pq-1}\bigg]+\frac{2(pq)(p+1)}{(pq-1)^2}\log(pq)\\ & \qquad +(j-1)\frac{(p+1)}{pq-1}\log(pq) -\frac{\log \widetilde{C}}{pq-1}
\end{align*}
Consequently, for an odd $j$ such that $j>\frac{\log \widetilde{C}}{(p+1)\log(pq)}-\frac{2(pq)}{pq-1}+1$, it holds
\begin{align} \label{lower bound log Dj}
\log D_j\geqslant (pq)^{\frac{j-1}{2}}\big(\log D_1-S_{p,q}(\infty)\big),
\end{align} where $S_{p,q}(\infty)\doteq \frac{2(pq)(p+1)}{(pq-1)^2}\log(pq)-\frac{\log \widetilde{C}}{pq-1}$.

Similarly, by using \eqref{lower bound Deltaj  n2}, it is possible to prove for an odd $j$ the following estimate:
\begin{align*}
\log \Delta_j & \geqslant (pq)^{\frac{j-1}{2}} \bigg[\log \Delta_{1}-\frac{2(pq)(q+1)}{(pq-1)^2}\log(pq)+\frac{\log \widetilde{K}}{pq-1}\bigg]+\frac{2(pq)(q+1)}{(pq-1)^2}\log(pq)\\ & \qquad +(j-1)\frac{(q+1)}{pq-1}\log(pq) -\frac{\log \widetilde{K}}{pq-1}.
\end{align*} Thus, for $j>\frac{\log \widetilde{K}}{(q+1)\log(pq)}-\frac{2(pq)}{pq-1}+1$ the last inequality implies 
\begin{align}\label{lower bound log Deltaj}
\log \Delta_j\geqslant (pq)^{\frac{j-1}{2}}\big(\log \Delta_1-\widetilde{S}_{p,q}(\infty)\big),
\end{align} where $\widetilde{S}_{p,q}(\infty)\doteq \frac{2(pq)(q+1)}{(pq-1)^2}\log(pq)-\frac{\log \widetilde{K}}{pq-1}$.  Let us set $j_0\doteq \big\lceil  \frac{1}{ \log(pq)} \max\{\frac{\log \widetilde{ C}}{p+1},\frac{\log\widetilde{K}}{q+1}\}-\frac{2pq}{pq-1}+1\big\rceil$, for the sake of brevity.
Combining the iterative inequality in \eqref{lower bound U j} and the lower bound in \eqref{lower bound log Dj}, for an odd $j>j_0$ and $t\geqslant 0$, employing \eqref{explicit epression a j} and  \eqref{explicit epression b j odd}, we arrive at 
\begin{align*}
U(t) & \geqslant \exp\Big((pq)^{\frac{j-1}{2}}\big(\log D_1- S_{p,q}(\infty)\big)\Big) (1+t)^{-a_j}t^{b_j} \\
& = \exp\Big((pq)^{\frac{j-1}{2}}\big(\log D_1- S_{p,q}(\infty)\big)\Big) (1+t)^{-\big(n+\frac{n-1}{2}p\big)(pq)^{\frac{j-1}{2}} +n}t^{\big(n+1+\frac{2(p+1)}{pq-1}\big)(pq)^{\frac{j-1}{2}} -\frac{2(p+1)}{pq-1}} \\
& = \exp\Big((pq)^{\frac{j-1}{2}}\big(\log D_1-\big(n+\tfrac{n-1}{2}p\big)\log(1+t)+\big(n+1+\tfrac{2(p+1)}{pq-1}\big)\log t- S_{p,q}(\infty)\big)\Big)  (1+t)^{n}t^{-\frac{2(p+1)}{pq-1}}.
\end{align*}
Consequently, for $t\geqslant 1$, using the inequality $\log 2t \geqslant \log(1+t)$, from the previous estimate we find
\begin{align}
U(t) & \geqslant \exp\Big((pq)^{\frac{j-1}{2}}J(t)\Big)   (1+t)^{n}t^{-\frac{2(p+1)}{pq-1}}, \label{lower bound U with J}
\end{align} where 
\begin{align}
J(t) & \doteq \log D_1+\big(\big(n+1+\tfrac{2(p+1)}{pq-1}\big)-\big(n+\tfrac{n-1}{2}p\big)\big)\log t -\big(n+\tfrac{n-1}{2}p\big)\log 2- S_{p,q}(\infty)\notag \\
& =  \log\bigg( D_1t^{\tfrac{pq+2p+1}{pq-1}-\tfrac{n-1}{2}p}\bigg)-\big(n+\tfrac{n-1}{2}p\big)\log 2- S_{p,q}(\infty). \label{def J(t)}
\end{align}
Let us point out that the power of $t$ in the above definition is positive if and only if $F(n,q,p)>0$.

In an analogous way, from \eqref{lower bound V j}, \eqref{lower bound log Deltaj}, \eqref{explicit epression alpha j} and \eqref{explicit epression beta j odd} we obtain for $t\geqslant 1$ and for an odd $j>j_0$
\begin{align}
V(t) & \geqslant \exp\Big((pq)^{\frac{j-1}{2}}\widetilde{J}(t)\Big)   (1+t)^{n}t^{-\frac{2(q+1)}{pq-1}}, \label{lower bound V with Jtilde}
\end{align} where 
\begin{align} \label{def Jtilde(t)}
\widetilde{J}(t) & \doteq  \log\bigg( \Delta_1t^{\tfrac{pq+2q+1}{pq-1}-\tfrac{n-1}{2}q}\bigg)-\big(n+\tfrac{n-1}{2}q\big)\log 2- \widetilde{S}_{p,q}(\infty)
\end{align} and in this case the power of $t$ is positive if and only if $F(n,p,q) >0$.

If $F(n,q,p)>0$, then, we can find $\varepsilon_0=\varepsilon_0(u_0,u_1,v_0,v_1,n,p,q,b_1,b_2,R)>0$ such that $$ \widehat{C}\varepsilon_0^{-F(n,q,p)^{-1}}\geqslant 1,$$ where  $\widehat{C}\doteq \Big(\tfrac{n(n+1)}{m_1(0)K_1}2^{n+\frac{n-1}{2}p}\exp(S_{p,q}(\infty))\Big)^{\frac{1}{pF(n,q,p)}}$. 
Therefore, for $\varepsilon\in (0,\varepsilon_0]$  and $t>  \widehat{C}\varepsilon^{-F(n,q,p)^{-1}}$ we have  $ t\geqslant 1$ and $J(t)>0$.  Letting $j\to \infty$ in \eqref{lower bound U with J}, the lower bound blows up and, consequently, $U$ may be finite only for $t\leqslant \widehat{C}\varepsilon^{-F(n,q,p)^{-1}}$. 

Analogously, in the other case $F(n,p,q)>0$, assuming that $$ \widehat{K}\varepsilon_0^{-F(n,p,p)^{-1}}\geqslant 1,$$ where the multiplicative constant in this case is given by $\widehat{K}\doteq \Big(\tfrac{n(n+1)}{m_2(0)C_1}2^{n+\frac{n-1}{2}q}\exp(\widetilde{S}_{p,q}(\infty))\Big)^{\frac{1}{qF(n,p,q)}}$, then, for any $\varepsilon\in (0,\varepsilon_0]$  and any $t>  \widehat{K}\varepsilon^{-F(n,p,q)^{-1}} $ we get $t\geqslant 1$ and $\widetilde{J}(t)>0$. Thus, as $j\to \infty$ in \eqref{lower bound V with Jtilde} the lower bound for $V(t)$ diverges. Also, $V$ may be finite only for $t\leqslant \widehat{K} \varepsilon^{-F(n,p,q)^{-1}}$. So, we proved Theorem \ref{Thm blowup iteration} and the estimate for the lifespan of the local solution given in \eqref{lifespan upper bound estimate}.

\subsection{Proof of Corollaries \ref{Corollary A}, \ref{Corollary B} and \ref{Corollary C}} \label{Subsection corollaries low dimensions} 

In this section we sketch how it is possible to modify the proof of Theorem \ref{Thm blowup iteration} in order to show the improvement of \eqref{lifespan upper bound estimate} as stated in Corollary  \ref{Corollary A}, in Corollary \ref{Corollary B} and in Corollary \ref{Corollary C}.

The first remark is that \eqref{Priori u^q} and \eqref{Priori v^p} can be improved in the case $n=1$ and in the case $n=2$ for exponents $p,q$ such that $1<p,q<2$, provided that the initial speeds for $u$ and $v$ are nontrivial (i.e., when the integrals of $u_1,v_1$ over $\mathbb{R}^n$ do not vanish). Indeed, as $U'(0)=\varepsilon \int_{\mathbb{R}^n} u_1(x) dx>0$ and $V'(0)=\varepsilon \int_{\mathbb{R}^n}v_1(x) dx>0$, since $U,V$ are convex functions, we get immediately
$U(t)\geqslant U'(0) \, t $ and $ V(t)\geqslant V'(0) \,t $. Consequently, using again the support condition for $u$ and $v$ and H\"older's inequality, we have
\begin{align*}
\int_{\mathbb{R}^n}|u(t,x)|^q dx & \geqslant \widetilde{C}_1  (1+t)^{-n(q-1)} (U(t))^q \geqslant \widetilde{C}_1   \varepsilon^q  \big(I[u_1]\big)^q(1+t)^{-n(q-1)} t^q , \\
\int_{\mathbb{R}^n}|v(t,x)|^p dx & \geqslant \widetilde{K}_1   (1+t)^{-n(p-1)} (V(t))^p \geqslant \widetilde{K}_1   \varepsilon^p  \big(I[ v_1]\big)^p(1+t)^{-n(p-1)} t^p,
\end{align*} where $\widetilde{C}_1,\widetilde{K}_1$ are suitable constants depending on $n,p,q,R$ and $I[f]\doteq \int_{\mathbb{R}^n} f(x)dx$. For large times, these lower bounds are stronger than \eqref{Priori u^q} and \eqref{Priori v^p} in the above mentioned cases. Hence, for the proofs of Corollaries  \ref{Corollary A}, \ref{Corollary B} and \ref{Corollary C} it is possible to follow faithfully the steps of the proof of Theorem  \ref{Thm blowup iteration} with few crucial modifications. If $n=1$ or $n=2 $ and $1<p<2$, then, \eqref{lower bound U j} in the base case is true for 
\begin{align*}
D_1 &\doteq \widetilde{K}_1 \big(I[ v_1]\big)^p \varepsilon^p, \quad  a_1\doteq p, \ \ b_1 \doteq (n-1)p, 
\end{align*}
and, similarly, if $n=1$ or $n=2 $ and $1<q<2$, then, \eqref{lower bound V j} in the base case is true for 
\begin{align*}
\Delta_1 &\doteq \widetilde{C}_1 \big(I[ u_1]\big)^q \varepsilon^q, \quad  \alpha_1\doteq q, \ \ \beta_1 \doteq (n-1)q.
\end{align*}
If $n=1$ or $n=2 $ and $1<p<2$, then, we can replace \eqref{def J(t)} by 
\begin{align*}
J(t) & = \log\bigg( D_1t^{b_1-a_1+\tfrac{2(p+1)}{pq-1}-n}\bigg)-(n+p)\log 2- S_{p,q}(\infty) = \log\bigg( D_1t^{p\,G(n,p,q)}\bigg)-(n+p)\log 2- S_{p,q}(\infty),
\end{align*} substituting the new values of $a_1,b_1$ in \eqref{explicit epression a j dep a1} and \eqref{explicit epression b j odd dep b1} instead of the ones used in Section \ref{Subsection iteration argument subcritical case}. Analogously, if $n=1$ or $n=2 $ and $1<q<2$, then, we can replace \eqref{def Jtilde(t)} by 
\begin{align*}
\widetilde{J}(t) & = \log\bigg( \Delta_1t^{\beta_1-\alpha_1+\tfrac{2(q+1)}{pq-1}-n}\bigg)-(n+q)\log 2-\widetilde{S}_{p,q}(\infty)= \log\bigg( \Delta_1t^{q\,G(n,q,p)}\bigg)-(n+q)\log 2- \widetilde{S}_{p,q}(\infty),
\end{align*} substituting now the new values of $\alpha_1,\beta_1$ in \eqref{explicit epression alpha j dep alpha1} and \eqref{explicit epression beta j odd dep beta1} in place of the ones used in Section \ref{Subsection iteration argument subcritical case}. Having in mind these changes, the proof of each corollary is a straightforward modification of the arguments used in Section \ref{Subsection iteration argument subcritical case}.

\section{Critical case: Proof of Theorem \ref{Thm critical case}} \label{Section critical case}

In this section we will prove Theorem \ref{Thm critical case}. The structure of the proof is organized as follows: in Section \ref{Subsection crit case: xi and eta functions} we recall the definition of certain auxiliary functions, which are necessary in order to introduce the functionals that we will estimate throughout the proof, and lower bound estimates for a fundamental system of solutions of the family of ODEs $\mathcal{L}_b y=0$, where $\mathcal{L}_b=\partial_t^2+b(t)\partial_ t-\lambda^2$ and $\lambda$ is a real parameter; moreover, using these estimates, we derive a couple of crucial estimates for the averages of the components of a local in time solution multiplied by one of the above cited auxiliary functions (these averages are actually the functionals whose dynamic we shall use to prove the blow-up result); then, in Section \ref{Subsection crit case: lower bounds estimates} we derive two coupled integral inequalities and lower bounds containing logarithmic terms for the functionals; in Section \ref{Subsection crit case: slicing method} we combine the lower bounds and the integral inequalities from Section \ref{Subsection crit case: lower bounds estimates} in order to prove a family of lower bound estimates via the slicing method; finally, in Section \ref{Subsection crit case: lifespan estimate} we use this sequence of lower bound estimates to proved the blow-up result and to derive the upper bound estimate for the lifespan of local in time solutions.

\subsection{Definition of the auxiliary functions} 
\label{Subsection crit case: xi and eta functions}

In this section we recall the definition of a pair of auxiliary functions from \cite{WakYor18}, which are necessary in order to introduce the time-dependent functionals that will be considered for the iteration argument. 

Let $r>-1$ be a parameter. Then, we introduce the functions
\begin{align*}
\xi_r(t,x) & \doteq  \int_0^{\lambda_0} e^{-\lambda(t+R)} \cosh (\lambda t) \, \Phi(\lambda x) \, \lambda^r \, d\lambda,\\
\eta_r(t,s,x) & \doteq  \int_0^{\lambda_0} e^{-\lambda(t+R)} \frac{\sinh (\lambda (t-s))}{\lambda(t-s)} \,\Phi(\lambda x) \,\lambda^r \, d\lambda,
\end{align*} where $\lambda_0$ is a fixed positive parameter and $\Phi$ is defined by \eqref{def eigenfunction laplace op}.

Some useful properties of $\xi_r$ and $\eta_r$ are stated in the following lemma, whose proof can be found in \cite[Lemma 3.1]{WakYor18}.

\begin{lem} \label{lemma eta and xi estimates}Let $n\geqslant 2$. There exist $\lambda_0>0$ such that the following properties hold:
\begin{itemize}
\item[\rm{(i)}] if $r>-1$, $|x|\leqslant R$ and $t\geqslant 0$, then, 
\begin{align*}
\xi_r(t,x) & \geqslant A_0, \\
\eta_r(t,0,x) & \geqslant B_0 \langle t\rangle^{-1};
\end{align*}
\item[\rm{(ii)}] if $r>-1$, $|x|\leqslant s+R$ and $t>s\geqslant 0$, then, 
\begin{align*}
\eta_r(t,s,x) & \geqslant B_1 \langle t\rangle^{-1} \langle s\rangle^{-r};
\end{align*}
\item[\rm{(iii)}] if $r>\frac{n-3}{2}$, $|x|\leqslant t+R$ and $t> 0$, then, 
\begin{align*}
\eta_r(t,t,x) & \leqslant B_2 \langle t\rangle^{-\frac{n-1}{2}} \langle t-|x| \rangle^{\frac{n-3}{2}-r}.
\end{align*}
\end{itemize}
Here $A_0$ and $B_k$, with $k=0,1,2$, are positive constants depending only on $\lambda_0$, $r$ and $R$ and we denote $\langle y\rangle \doteq 3+|y|$.
\end{lem}

\begin{rem} Even though in \cite{WakYor18} the previous lemma is stated requiring $r>0$ in {\rm(i)} and {\rm(ii)}, the proof provided in that paper is  valid for any $r>-1$ as well.
\end{rem}


\begin{lem}\label{lemma y1,y2} Let $\lambda$ be a positive real parameter and let $b\in \mathcal{C}^1([0,\infty))\cap L^1([0,\infty))$ be a nonnegative function. We introduce the differential operators
\begin{align*}
\mathcal{L}_b\doteq \partial_t^2 +b(t)\partial_t -\lambda^2, \quad \mathcal{L}_b^* \doteq \partial_s^2 -\partial_s b(s) -\lambda^2
\end{align*} and the fundamental system of solutions 
$y_j=y_j(t,s;\lambda, b)$, with $j=1,2$, such that
\begin{align*}
\mathcal{L}_b y_1(t,s;\lambda, b)=0, \ \ y_1(s,s;\lambda, b)=1,  \ \ \partial_t y_1(s,s;\lambda, b)=0; \\ 
\mathcal{L}_b y_2(t,s;\lambda, b)=0, \ \ y_2(s,s;\lambda, b)=0,  \ \ \partial_t y_1(s,s;\lambda, b)=1.
\end{align*} Then, $\{y_1,y_2\}$ depends continuously on $\lambda$ and satisfies for $t\geqslant s \geqslant 0$ the following estimates:
\begin{align*}
& \mathrm{(i)} \ \ y_1(t,s;\lambda, b)\geqslant e^{-\| b\|_{L^1}} \cosh \lambda(t-s), \\
& \mathrm{(ii)} \ \ y_2(t,s;\lambda, b)\geqslant e^{-2\| b\|_{L^1}} \frac{\sinh \lambda(t-s)}{\lambda}. 
\end{align*} Moreover,
\begin{align*}
& \mathrm{(iii)} \ \ \mathcal{L}_b^* y_2(t,s;\lambda, b)=0, \\
& \mathrm{(iv)} \ \ y_1(t,0;\lambda, b)= b(0) y_2(t,0;\lambda, b)-\partial_s y_2(t,0;\lambda, b), \\
& \mathrm{(v)} \ \ \partial_s y_2(t,t;\lambda, b)= -1. 
\end{align*}
\end{lem}

\begin{proof}
See Lemma 2.3 in \cite{WakYor18damp}. In particular, (v) follows by the first condition in \cite[relation (4.7)]{WakYor18damp}.
\end{proof}

\begin{prop} \label{prop lower bounds critical case} Let $b_1,b_2\in \mathcal{C}^1([0,\infty))\cap L^1([0,\infty))$ and let $u_0,v_0\in H^1(\mathbb{R}^n)$ and $u_1,v_1\in  L^2(\mathbb{R}^n)$ be nonnegative, pairwise nontrivial and compactly supported in $B_R$.  Let $(u,v)$ be an energy solution to \eqref{weakly coupled system} on $[0,T)$ according to Definition \ref{def energ sol intro} satisfying \eqref{support condition solution}. Then, the following estimates hold:
\begin{align}
\int_{\mathbb{R}^n}u(t,x)\, \eta_{r_1}(t,t,x) \, dx &\geqslant  e^{-\| b_1\|_{L^1}} \varepsilon \int_{\mathbb{R}^n}u_0(x) \xi_{r_1}(t,x) \, dx +  e^{-2\| b_1\|_{L^1}} \varepsilon t \int_{\mathbb{R}^n}u_1(x) \eta_{r_1}(t,0,x) \, dx\notag\\ & \quad + e^{-2\| b_1\|_{L^1}}  \int_0^t (t-s)\int_{\mathbb{R}^n}|v(s,x)|^p\eta_{r_1}(t,s,x) \, dx \, ds, \label{fund ineq mathcal U} \\
\int_{\mathbb{R}^n}v(t,x)\, \eta_{r_2}(t,t,x) \, dx &\geqslant e^{-\| b_2\|_{L^1}} \varepsilon \int_{\mathbb{R}^n}v_0(x) \xi_{r_2}(t,x) \, dx +  e^{-2\| b_2\|_{L^1}} \varepsilon t \int_{\mathbb{R}^n}v_1(x) \eta_{r_2}(t,0,x) \, dx \notag\\ & \quad + e^{- 2\| b_2\|_{L^1}}  \int_0^t (t-s)\int_{\mathbb{R}^n}|u(s,x)|^q\eta_{r_2}(t,s,x) \, dx \, ds \label{fund ineq mathcal V}
\end{align} for $r_1,r_2>-1$ and any $t\in (0,T)$.
\end{prop}

\begin{proof}
Thanks to \eqref{support condition solution} 
we have that $u(t,\cdot),v(t,\cdot)$ have compact support in $B_{R+t}$ for any $t\geqslant 0$. Therefore, we may employ \eqref{def u} and \eqref{def v} also for noncompactly supported test function. Moreover, by using a density argument we can weaken the regularity for the test functions in Definition \ref{def energ sol intro}.
%
Consequently, we may choose as test functions $$\phi=\phi(s,x)=\Phi(\lambda x)\, y_2(t,s;\lambda,b_1), \qquad \psi=\psi(s,x)=\Phi(\lambda x)\, y_2(t,s;\lambda,b_2),$$ where $\Phi$ is defined by \eqref{def eigenfunction laplace op}. As $\Phi$ is an eigenfunction of the Laplace operator and $y_2(t,s;\lambda,b_1)$, $ y_2(t,s;\lambda,b_2)$ solve $\mathcal{L}^*_{b_1}y=0$ and $\mathcal{L}^*_{b_2}y=0$, respectively, we get that $\phi$ and $\psi$ satisfy
\begin{align*}
&\phi_{ss}-\Delta \phi -\partial_s(b_1(s) \phi)=0 \,  \qquad b_1(0)\phi(0,x)-\phi_s(0,x)= \Phi(\lambda x) y_1(t,0; \lambda,b_1)  \ \ \mbox{and} \ \ \phi_s(t,x)= -\Phi(\lambda x),\\
&\psi_{ss}-\Delta \psi -\partial_s(b_2(s) \psi)=0 \qquad b_2(0)\psi(0,x)-\psi_s(0,x)= \Phi(\lambda x) y_1(t,0; \lambda,b_2)  \ \ \mbox{and} \ \ \psi_s(t,x)= -\Phi(\lambda x),
\end{align*} where we employed (iv) and (v) from Lemma \ref{lemma y1,y2} to get the relations for the values of $\phi$ and $\psi$ at $s=0,t$.

Let us prove \eqref{fund ineq mathcal U}. Using the above defined $\phi$ in \eqref{def u weak} and its properties, we get
\begin{align*}
\int_{\mathbb{R}^n} u(t,x) \Phi(\lambda x) \,dx &= \varepsilon y_1(t,0;\lambda,b_1)\int_{\mathbb{R}^n} u_0(x) \Phi(\lambda x) \, dx + \varepsilon y_2(t,0;\lambda,b_1)\int_{\mathbb{R}^n} u_1(x) \Phi(\lambda x) \, dx \\ & \quad +\int_0^t   y_2(t,s;\lambda,b_1)\int_{\mathbb{R}^n}|v(s,x)|^p \Phi(\lambda x) \, dx,
\end{align*} where we used also the condition $\phi(t,x)=0$, which follows immediately from the initial values of $y_2(t,s;\lambda,b_1)$ prescribed in the statement of Lemma \ref{lemma y1,y2}. Using the estimates from below (i) and (ii) in Lemma \ref{lemma y1,y2}, we obtain from the previous relation 
\begin{align*}
\int_{\mathbb{R}^n} u(t,x) \Phi(\lambda x) \,dx & \geqslant \varepsilon e^{-\| b_1\|_{L^1}} \cosh \lambda t \int_{\mathbb{R}^n} u_0(x) \Phi(\lambda x) \, dx + \varepsilon  e^{-2\| b_1\|_{L^1}} \frac{\sinh \lambda t}{\lambda} \int_{\mathbb{R}^n} u_1(x) \Phi(\lambda x) \, dx \\ & \quad +  e^{-2\| b_1\|_{L^1}} \int_0^t  \frac{\sinh \lambda(t-s)}{\lambda} \int_{\mathbb{R}^n}|v(s,x)|^p \Phi(\lambda x) \, dx.
\end{align*} Multiplying both sides of the last inequality by $e^{-\lambda(t+R)}\lambda^{r_1}$, integrating with respect to $\lambda$ over $[0,\lambda_0]$ and applying Tonelli's theorem, we get finally \eqref{fund ineq mathcal U}. In order to prove \eqref{fund ineq mathcal V}, it is sufficient to repeat the above steps after plugging the prescribed $\psi$ function in \eqref{def v weak}. Hence, the proof is complete.
\end{proof}

\subsection{Lower bound estimates} \label{Subsection crit case: lower bounds estimates}

Hereafter until the end of Section  \ref{Section critical case}, we will assume that $u_0,u_1,v_0,v_1$ satisfy the assumptions from the statement of Theorem \ref{Thm critical case}. Furthermore, without loss of generality we assume that $(p,q)$ satisfies the critical condition $F(n,p,q)=0$, because the case $F(n,q,p)=0$ is completely symmetric, assumed the switch of $p,q$ and $u,v$, respectively. Let $(u,v)$ be an energy solution of \eqref{weakly coupled system} on $[0,T)$. We introduce the following time-dependent functionals
\begin{equation}\label{defn functional crit case}
\begin{split}
\mathcal{U}(t)&\doteq \int_{\mathbb{R}^n}u(t,x) \, \eta_{r_1}(t,t,x) \, dx \, , \\
\mathcal{V}(t)&\doteq \int_{\mathbb{R}^n}v(t,x) \, \eta_{r_2}(t,t,x) \, dx \, .
\end{split}
\end{equation} Let us point out that we will prescribe in the next proposition the exact assumptions for the parameters $r_1,r_2$. From Proposition \ref{prop lower bounds critical case} it follows immediately the positiveness of the functionals $\mathcal{U},\mathcal{V}$.

The next step is to derive two integral inequalities involving $\mathcal{U}$ and $\mathcal{V}$ in a ``coupled way'', and, as we have just mentioned, this goal will somehow fix the range for $(r_1,r_2)$. Let us point out explicitly that the case $p>q$ and the case $p=q$ (see Remark \ref{rmk p and q relation} below) will be treated separately with a different choice of the pair $(r_1,r_2)$ (which will correspond to a different frame for the iteration scheme). 

\begin{rem} \label{rmk p and q relation} Since we assume that $(p,q)$ satisfies $F(n,p,q)=0\leqslant F(n,q,p)$ it may be either $F(n,q,p)<0$ or $F(n,q,p)=0$. Due to the monotonicity of the function $f=f(p)=p-p^{-1}$ for $p>1$, in the first case we are in the case $p>q$, while in the latter case we have $p=q$. Moreover, the condition $F(n,p,p)=0$ it equivalent to require $p=p_0(n)$, so that $F(n,p,q)=F(n,q,p)=0$ corresponds to the limit case $p=q=p_0(n)$.
\end{rem}

\begin{prop} \label{prop integral ineq critic case} Let us assume that $r_1,r_2$ are given parameters satisfying $r_1=\frac{n-1}{2}-\frac{1}{q}$ and 
\begin{itemize}
\item  $r_2>\frac{n-1}{2}-\frac{1}{p}$ if $p>q$;
\item $r_2=\frac{n-1}{2}-\frac{1}{p}$ if $p=q$.
\end{itemize}  Let $\mathcal{U},\mathcal{V}$ be the functionals defined by \eqref{defn functional crit case}. Then, there exist positive constants $C$ and $K$ depending on $n,p,q,R,b_1,b_2$ such that for any $t\geqslant 0$ the following estimates hold:
\begin{align}
\mathcal{U}(t) &\geqslant C \langle t\rangle^{-1} \int_0^t(t-s) \langle s \rangle^{\frac{n-1}{2}+1+\frac{1}{q}+(r_2+1-n)p}  (\mathcal{V}(s))^p \, ds \, ,  \label{II2}\\
\mathcal{V}(t) &\geqslant K \langle t\rangle^{-1} \int_0^t(t-s) \langle s \rangle^{-r_2-\frac{n-1}{2}q+n-1} \big(\log\langle s \rangle\big)^{-(q-1)} (\mathcal{U}(s))^q \, ds \,  \label{II1}
\end{align}
 for $p>q$ and
\begin{align}
\mathcal{U}(t) &\geqslant C \langle t\rangle^{-1} \int_0^t(t-s) \langle s \rangle^{-1}  \big(\log\langle s \rangle\big)^{-(p-1)} (\mathcal{V}(s))^p \, ds \, ,  \label{II2 crit}\\
\mathcal{V}(t) &\geqslant K \langle t\rangle^{-1} \int_0^t(t-s) \langle s \rangle^{-1} \big(\log\langle s \rangle\big)^{-(q-1)} (\mathcal{U}(s))^q \, ds \,  \label{II1 crit}
\end{align} for  $p=q$.
\end{prop}

\begin{proof} 
For the proof of this result we will follow the main ideas of Proposition 4.2 in \cite{WakYor18}. Let us begin with the proof in the case $p>q$. 
By H\"older's inequality and the support property for $v(s,\cdot)$, we obtain
\begin{align}
\mathcal{V}(s) \leqslant \bigg(\int_{\mathbb{R}^n}|v(s,x)|^p \eta_{r_1}(t,s,x) \, dx\bigg)^{\frac{1}{p}} \Bigg(\int_{B_{s+R}}\frac{\eta_{r_2}(s,s,x)^{p'}}{\eta_{r_1}(t,s,x)^{\frac{p'}{p}}}\, dx\Bigg)^{\frac{1}{p'}}. \label{holder + compact supp v}
\end{align} We estimate the second factor on the right hand side in the last inequality. By (ii) and (iii) in Lemma \ref{lemma eta and xi estimates} (note that, according to our choice in the statement of this proposition, both conditions $r_1,r_2 >\frac{n-3}{2}$ and $r_1,r_2>-1$ are always fulfilled), we obtain
\begin{align*}
\int_{B_{s+R}}\frac{\eta_{r_2}(s,s,x)^{p'}}{\eta_{r_1}(t,s,x)^{\frac{p'}{p}}}\, dx & \lesssim \langle t\rangle^{\frac{p'}{p}}\langle s\rangle^{-\frac{n-1}{2}p'+\frac{p'}{p}r_1}\int_{B_{s+R}}\langle s-|x|\rangle^{(\frac{n-3}{2}-r_2)p'} \, dx \\
& \lesssim \langle t\rangle^{\frac{p'}{p}}\langle s\rangle^{-\frac{n-1}{2}p'+\frac{p'}{p}r_1+n+(\frac{n-3}{2}-r_2)p'},
\end{align*} where in the  last step we used the assumption on $r_2$ which is equivalent to require a power smaller than $-1$ for the term $\langle s-|x|\rangle$ in the integral. Combining  \eqref{fund ineq mathcal U}, \eqref{holder + compact supp v} and the previous estimate, we find
\begin{align*}
\mathcal{U}(t) &\gtrsim \int_0^t (t-s) \int_{\mathbb{R}^n}|v(s,x)|^p \eta_{r_1}(t,s,x) \, dx \, ds \gtrsim  \int_0^t (t-s) (\mathcal{V}(s))^{p} \langle t\rangle^{-1}\langle s\rangle^{\frac{n-1}{2}p-r_1-n(p-1)-(\frac{n-3}{2}-r_2)p} \, ds   \\ 
  & \gtrsim \langle t\rangle^{-1}  \int_0^t (t-s) (\mathcal{V}(s))^{p} \langle s\rangle^{\frac{n-1}{2}+1+\frac{1}{q}+(r_2+1-n)p} \, ds.
\end{align*} 

Let us prove now \eqref{II1}. Analogously to \eqref{holder + compact supp v}, we get
\begin{align}
\mathcal{U}(s) \leqslant \bigg(\int_{\mathbb{R}^n}|u(s,x)|^q \eta_{r_2}(t,s,x) \, dx\bigg)^{\frac{1}{q}} \Bigg(\int_{B_{s+R}}\frac{\eta_{r_1}(s,s,x)^{q'}}{\eta_{r_2}(t,s,x)^{\frac{q'}{q}}}\, dx\Bigg)^{\frac{1}{q'}}. \label{holder + compact supp u}
\end{align} Employing again (ii) and (iii) in Lemma \ref{lemma eta and xi estimates} and thanks the choice of the parameter $r_1$, we arrive at
\begin{align*}
\int_{B_{s+R}}\frac{\eta_{r_1}(s,s,x)^{q'}}{\eta_{r_2}(t,s,x)^{\frac{q'}{q}}}\, dx & \lesssim \langle t\rangle^{\frac{q'}{q}}\langle s\rangle^{-\frac{n-1}{2}q'+\frac{q'}{q}r_2}\int_{B_{s+R}}\langle s-|x|\rangle^{(\frac{n-3}{2}-r_1)q'} \, dx \\
& = \langle t\rangle^{\frac{q'}{q}}\langle s\rangle^{-\frac{n-1}{2}q'+\frac{q'}{q}r_2}\int_{B_{s+R}}\langle s-|x|\rangle^{-1} \, dx \lesssim \langle t\rangle^{\frac{q'}{q}}\langle s\rangle^{-\frac{n-1}{2}q'+\frac{q'}{q}r_2+n-1} \log\langle s \rangle.
\end{align*} If we combine \eqref{fund ineq mathcal V}, \eqref{holder + compact supp u} and the last estimate, we have
\begin{align}
\mathcal{V}(t) & \gtrsim \int_0^t (t-s) \int_{\mathbb{R}^n}|u(s,x)|^q \eta_{r_2}(t,s,x) \, dx \, ds  \notag \\ & \gtrsim \int_0^t (t-s)  (\mathcal{U}(s))^{q} \langle t\rangle^{-1}\langle s\rangle^{\frac{n-1}{2}q-r_2-(n-1)(q-1)} \big(\log\langle s \rangle\big)^{-(q-1)} \, ds \notag\\
& =  \langle t\rangle^{-1} \int_0^t (t-s) (\mathcal{U}(s))^{q} \langle s\rangle^{-r_2-\frac{n-1}{2}q+n-1} \big(\log\langle s \rangle\big)^{-(q-1)} \, ds. \label{II1 p>q proof}
\end{align} In the case $p=q=p_0(n)$, if we plug the value of $r_2$ in \eqref{II1 p>q proof}, then, thanks to $-\frac{n-1}{2}(q-1)+\frac{1}{q}=-1$, we get immediately \eqref{II1 crit}. Due to symmetry reasons the proof of \eqref{II2 crit} is totally analogous. This completes the proof of the proposition.
\end{proof}

The integral inequalities derived in the last proposition will play a fundamental role in the iteration argument. However, in order to start with this iteration argument we have to derive a lower bound for the functional $\mathcal{U}$ containing a logarithmic term. 
For this purpose we will combine the lower bounds for the nonlinearities that we have shown in the subcritical case in Lemma \ref{lemma lower bound nonlinearities} with the estimates from Lemma \ref{lemma eta and xi estimates} and Proposition \ref{prop integral ineq critic case}.

\begin{lem}\label{lower bound functionals log terms crit} Let $p,q>1$ satisfy $F(n,p,q)=0$. Then, for any $t\geqslant \frac{3}{2}$ the following estimates hold:
\begin{align*}
\mathcal{U}(t) &\geqslant \widetilde{C} \varepsilon^{pq} \log\big(\tfrac{2t}{3}\big) \qquad \mbox{if} \ \ p>q, \\
\mathcal{U}(t) &\geqslant \widetilde{C} \varepsilon^{p} \log\big(\tfrac{2t}{3}\big) \ \qquad \mbox{if} \ \ p=q, 
\end{align*} where $\widetilde{C}$ is a positive constant depending on $n,p,q,u_0,u_1,v_0,v_1,b_1,b_2,R$.
\end{lem}

\begin{proof}
 From \eqref{fund ineq mathcal U}, estimate (ii) in Lemma \ref{lemma eta and xi estimates}, \eqref{Priori v^p} and the definition of $r_1$ we obtain for $t\geqslant 1$
\begin{align}
\mathcal{U}(t) & \gtrsim \int_0^t (t-s)\int_{\mathbb{R}^n}|v(s,x)|^p\eta_{r_1}(t,s,x) \, dx \, ds  \gtrsim \langle  t\rangle^{-1} \int_0^t (t-s) \langle s\rangle^{-r_1}\int_{\mathbb{R}^n}|v(s,x)|^p \, dx \, ds \notag \\
&  \gtrsim \varepsilon^p \langle  t\rangle^{-1} \int_0^t (t-s) \langle s\rangle^{-r_1+n-1-\frac{n-1}{2}p} \, ds =\varepsilon^p \langle  t\rangle^{-1} \int_0^t (t-s) \langle s\rangle^{\frac{n-1}{2}+\frac{1}{q}-\frac{n-1}{2}p} \, ds. \label{inter est 1 U}
\end{align} 

Similarly, by \eqref{fund ineq mathcal V} and \eqref{Priori u^q} we get  for $t\geqslant 1$
\begin{align}
\mathcal{V}(t) 
&  \gtrsim \varepsilon^q \langle  t\rangle^{-1} \int_0^t (t-s) \langle s\rangle^{-r_2+n-1-\frac{n-1}{2}q} \, ds . \label{inter est 1 V}
\end{align}

In the special case $p=q=p_0(n)$, the power of $\langle s \rangle$ in the integral in the right hand side of \eqref{inter est 1 U} is $-1$. Hence, we may estimate for $t\geqslant \frac{3}{2}$
\begin{align}
\int_0^t (t-s) \langle s\rangle^{-1} \, ds & \gtrsim \int_1^t \frac{t-s}{s} \, ds = \int_1^t \log s \, ds \geqslant \int_\frac{2t}{3}^t \log s \, ds \gtrsim t \log \big(\tfrac{2t}{3}\big) \gtrsim \langle t\rangle \log \big(\tfrac{2t}{3}\big). \label{inter est 2}
\end{align} Combining \eqref{inter est 1 U} and \eqref{inter est 2}, we get the desired estimate in the case $p=q$. In the case $q>p$, keeping on the estimate in \eqref{inter est 1 V}, we find  for $t\geqslant 1$
\begin{align*}
\mathcal{V}(t) 
&  \gtrsim \varepsilon^q \langle  t\rangle^{-1-r_2-\frac{n-1}{2}q} \int_{\tfrac{t}{2}}^t (t-s) \langle s\rangle^{n-1} \, ds \gtrsim \varepsilon^q \langle  t\rangle^{-1-r_2-\frac{n-1}{2}q}  \langle \tfrac{t}{2}\rangle^{n-1} \int_{\tfrac{t}{2}}^t (t-s) \, ds \gtrsim \varepsilon^q \langle  t\rangle^{-r_2-\frac{n-1}{2}q+n} . 
\end{align*} Now we plug the previous lower bound in \eqref{II2}, after shrinking the domain of integration, and we arrive for $t\geqslant \frac{3}{2}$ at
\begin{align*}
\mathcal{U}(t) & \gtrsim \varepsilon^{pq} \langle t\rangle^{-1} \int_1^t(t-s) \langle s \rangle^{\frac{n-1}{2}+1+\frac{1}{q}+(r_2+1-n)p-r_2 p-\frac{n-1}{2}pq+n p}   \, ds \\
& = \varepsilon^{pq} \langle t\rangle^{-1} \int_1^t(t-s) \langle s \rangle^{-\frac{n-1}{2}(pq-1)+p+1+\frac{1}{q}}  \, ds   = \varepsilon^{pq} 
 \langle t\rangle^{-1}\int_1^t(t-s) \langle s \rangle^{-1}  \, ds  \gtrsim \varepsilon^{pq}   \log\big(\tfrac{2t}{3}\big),
\end{align*} where we used the condition $F(n,p,q)=0$ in the second last step and again \eqref{inter est 2} in the last step. This concludes the proof.
\end{proof}

\subsection{Iteration argument via slicing method} 
\label{Subsection crit case: slicing method}

In this section we apply the so-called slicing method, which has been introduced for the first time in \cite{AKT00}, in order to prove a family of lower bound estimates for $\mathcal{U}$. Let us introduce the sequence $\{\ell_j\}_{j\in \mathbb{N}}$, where $\ell_j\doteq 2-2^{-(j+1)}$. The goal of this iteration method is to prove 
\begin{align}
\mathcal{U}(t)\geqslant C_j (\log\langle t\rangle)^{-b_j} \bigg(\log\bigg(\frac{t}{\ell_{2j}}\bigg)\bigg)^{a_j} \qquad \mbox{for} \ \ t\geqslant \ell_{2j} \ \ \mbox{and for any} \ \ j\in \mathbb{N}, \label{iter ineq crit case}
\end{align} where $\{C_j\}_{j\in\mathbb{N}}$,  $\{a_j\}_{j\in\mathbb{N}}$ and  $\{b_j\}_{j\in\mathbb{N}}$ are sequences of nonnegative real numbers that we shall determine afterwards.
For $j=0$ we know that \eqref{iter ineq crit case} is true thanks to Lemma \ref{lower bound functionals log terms crit} with 
\begin{align*}
C_0\doteq \begin{cases} \widetilde{C} \varepsilon^{pq} & \mbox{if} \ \ p>q , \\ \widetilde{C} \varepsilon^{p} & \mbox{if} \ \ p=q, \end{cases} \qquad a_0\doteq 1 \quad \mbox{and} \quad b_0\doteq 0.
\end{align*}

We are going to prove the validity of \eqref{iter ineq crit case} by using an inductive proof. As we have already remarked the validity of the base case, it remains to prove the inductive step. Let us assume that \eqref{iter ineq crit case} holds for $j\geqslant 1$, we want to prove it now for $j+1$. Because of the different frame in \eqref{II2}-\eqref{II1} and in \eqref{II2 crit}-\eqref{II1 crit}, we shall consider separately the cases $p>q$ and $p=q$.

\subsubsection*{Case $p>q$}

Combining \eqref{II1} and \eqref{iter ineq crit case} for $j$, for any $s\geqslant \ell_{2j+1}$ it follows
\begin{align}
\mathcal{V}(s) &\geqslant K \langle s\rangle^{-1} \int_{\ell_{2j}}^s(s-\tau) \langle \tau \rangle^{-r_2-\frac{n-1}{2}q+n-1} \big(\log\langle \tau \rangle\big)^{-(q-1)} (\mathcal{U}(\tau))^q \, d\tau \notag \\
&\geqslant K C_j^q \langle s\rangle^{-1} \int_{\ell_{2j}}^s(s-\tau) \langle \tau \rangle^{-r_2-\frac{n-1}{2}q+n-1} \big(\log\langle \tau \rangle\big)^{-(q-1)-b_j q} \left(\log\left(\tfrac{\tau}{\ell_{2j}}\right)\right)^{a_j q}\, d\tau \notag \\
&\geqslant K C_j^q \big(\log\langle s \rangle\big)^{-(q-1)-b_j q} \langle s \rangle^{-1-r_2-\frac{n-1}{2}q}\int_{\ell_{2j}}^s(s-\tau) \langle \tau \rangle^{n-1} \left(\log\left(\tfrac{\tau}{\ell_{2j}}\right)\right)^{a_j q}\, d\tau. \label{estimate mathcalV inter 1} 
\end{align} 
We can estimate now the integral in the last line of the previous chain of inequalities as follows
\begin{align}
\int_{\ell_{2j}}^s(s-\tau) \langle \tau \rangle^{n-1} \left(\log\left(\tfrac{\tau}{\ell_{2j}}\right)\right)^{a_j q}\, d\tau & \geqslant \int_{\tfrac{\ell_{2j} s}{\ell_{2j+1}}}^s(s-\tau)\, \tau^{n-1} \left(\log\left(\tfrac{\tau}{\ell_{2j}}\right)\right)^{a_j q}\, d\tau \notag \\
& \geqslant \left(\tfrac{\ell_{2j}}{\ell_{2j+1}}\right)^{n-1} s^{n-1}\left(\log\left(\tfrac{s}{\ell_{2j+1}}\right)\right)^{a_j q}\int_{\tfrac{\ell_{2j} s}{\ell_{2j+1}}}^s(s-\tau) \, d\tau \notag  \\
& \geqslant  2^{-n}  \left(1-\tfrac{\ell_{2j}}{\ell_{2j+1}}\right)^2 s^{n+1}\left(\log\left(\tfrac{s}{\ell_{2j+1}}\right)\right)^{a_j q} \notag \\
&\geqslant 2^{-4j-3n-8} \langle s \rangle^{n+1} \left(\log\left(\tfrac{s}{\ell_{2j+1}}\right)\right)^{a_j q}, \label{estimate mathcalV inter 2}
\end{align}
where we used the relation $\langle y \rangle\geqslant y \geqslant \frac{1}{4} \langle y \rangle$ for $y\geqslant 1$ in the first and last inequality. Moreover, in the first line we might restrict the domain of integration since $\ell_{2j}<\ell_{2j+1}$, in the third one we used the inequality $2\ell_{2j} \geqslant \ell_{2j+1}$ and, finally, we employed the condition $1-\frac{\ell_k}{\ell_{k+1}}\geqslant 2^{-(k+3)}$.
 Hence, plugging \eqref{estimate mathcalV inter 2} in \eqref{estimate mathcalV inter 1}, for any $s\geqslant \ell_{2j+1}$ we find
\begin{align*}
\mathcal{V}(s) &\geqslant 2^{-4j-3n-8} K C_j^q\big(\log\langle s \rangle\big)^{-(q-1)-b_j q} \langle s \rangle^{-r_2-\frac{n-1}{2}q+n} \left(\log\left(\tfrac{s}{\ell_{2j+1}}\right)\right)^{a_j q}.
\end{align*}
Next we use the previous lower bound for $\mathcal{V}(s)$ in \eqref{II2}, so that for $t\geqslant \ell_{2j+2}$ we have
\begin{align} 
\mathcal{U}(t) &\geqslant  2^{-(4j+3n+8)p} C K^p C_j^{pq}  \langle t\rangle^{-1} \! \! \int_{\ell_{2j+1}}^t \!(t-s) \langle s \rangle^{-\frac{n-1}{2}(pq-1)+p+1+\frac{1}{q}}  \! \big(\log\langle s \rangle\big)^{-p(q-1)-b_j pq} \left(\log\left(\tfrac{s}{\ell_{2j+1}}\right)\right)^{a_j pq}  ds \notag\\
&\geqslant  2^{-(4j+3n+8)p} C K^p C_j^{pq}  \big(\log\langle t \rangle\big)^{-p(q-1)-b_j pq} \langle t\rangle^{-1} \int_{\ell_{2j+1}}^t(t-s) \langle s \rangle^{-1}  \left(\log\left(\tfrac{s}{\ell_{2j+1}}\right)\right)^{a_j pq} \, ds \notag \\
&\geqslant  2^{-(4j+3n+8)p-2} C K^p C_j^{pq} \big(\log\langle t \rangle\big)^{-p(q-1)-b_j pq} \langle t\rangle^{-1} \int_{\ell_{2j+1}}^t\tfrac{t-s}{s} \left(\log\left(\tfrac{s}{\ell_{2j+1}}\right)\right)^{a_j pq} \, ds, \label{estimate mathcalU inter 1} 
\end{align} where in the second inequality the condition $F(n,p,q)=0$ implies that the exponent of the factor $\langle s \rangle$ is exactly $-1$. Using integration by parts, we may estimate the last integral in the following way:
\begin{align}
\int_{\ell_{2j+1}}^t\tfrac{t-s}{s} \left(\log\left(\tfrac{s}{\ell_{2j+1}}\right)\right)^{a_j pq} \, ds 
& =  (a_jpq+1)^{-1}  \int_{\ell_{2j+1}}^t \left(\log\left(\tfrac{s}{\ell_{2j+1}}\right)\right)^{a_j pq+1} \, ds \notag \\
& \geqslant   (a_jpq+1)^{-1}  \int_{\tfrac{\ell_{2j+1}t}{\ell_{2j+2}}}^t \left(\log\left(\tfrac{s}{\ell_{2j+1}}\right)\right)^{a_j pq+1} \, ds \notag \\
& \geqslant  (a_jpq+1)^{-1} \left(1-\tfrac{\ell_{2j+1}}{\ell_{2j+2}}\right)  t \, \left(\log\left(\tfrac{t}{\ell_{2j+2}}\right)\right)^{a_j pq+1}\notag  \\
& \geqslant  2^{-2j-6}  (a_jpq+1)^{-1} \langle t \rangle\left(\log\left(\tfrac{t}{\ell_{2j+2}}\right)\right)^{a_j pq+1}, \label{estimate mathcalU inter 2}
\end{align} where in the second step it is possible to shrink the domain of integration due to $t\geqslant \ell_{2j+2}$. Also, combining \eqref{estimate mathcalU inter 1} and \eqref{estimate mathcalU inter 2}, we have
\begin{align*}
\mathcal{U}(t) &\geqslant 2^{-(2p+1)2j-(3n+8)p-8} C K^p C_j^{pq} (a_jpq+1)^{-1} \big(\log\langle t \rangle\big)^{-p(q-1)-b_j pq} \left(\log\left(\tfrac{t}{\ell_{2j+2}}\right)\right)^{a_j pq+1}.
\end{align*}
Therefore, if we put 
\begin{align} C_{j+1}\doteq 2^{-(2p+1)2j-(3n+8)p-8} C K^p C_j^{pq} (a_jpq+1)^{-1} , \ \  a_{j+1}\doteq a_j pq+1\ \ \mbox{ and} \  \ b_{j+1} \doteq p(q-1)+b_j pq,
\end{align} then, we proved \eqref{iter ineq crit case} for $j+1$ in the case $p>q$.

Let us determine explicitly the expressions of $a_j$ and $b_j$. By using recursively the above relations, we find
\begin{align}
a_j  & = a_{j-1} pq+1  = a_0(pq)^j +\sum_{k=0}^{j-1} (pq)^k  =(pq)^j  +\tfrac{(pq)^j-1}{pq-1}= \tfrac{(pq)^{j+1}-1}{pq-1}\, , \label{def aj crit case}\\
b_j & = p(q-1)+b_{j-1} pq = p(q-1)\sum_{k=0}^{j-1} (pq)^k+b_{0} (pq)^j = \tfrac{p(q-1)}{pq-1}\big((pq)^j-1\big). \label{def bj crit case}
\end{align} In particular, $$a_{j-1} pq+1=\tfrac{(pq)^{j+1}-1}{pq-1}\leqslant \tfrac{pq}{pq-1}(pq)^j, $$ 
which implies in turn
\begin{align}\label{inter est 3}
C_j\geqslant M \Theta^{-j} C_{j-1}^{pq},
\end{align} where  $\Theta\doteq 2^{2(2p+1)}pq$ and $M\doteq 2^{-(3n+4)p-6}CK^p \tfrac{(pq-1)}{(pq)}$. 

\subsubsection*{Case $p=q$}

We have to modify slightly the procedure seen in the case $p>q$, by using \eqref{II2 crit}-\eqref{II1 crit} in place of \eqref{II2}-\eqref{II1}. Using \eqref{II1 crit} and \eqref{iter ineq crit case} for $j$, for any $s\geqslant \ell_{2j+1}$ we have
%
\begin{align}
\mathcal{V}(s) &\geqslant K \langle s\rangle^{-1} \int_{\ell_{2j}}^s(s-\tau) \langle \tau \rangle^{-1} \big(\log\langle \tau \rangle\big)^{-(q-1)} (\mathcal{U}(\tau))^q \, d\tau \notag \\
&\geqslant K C_j^q \langle s\rangle^{-1} \int_{\ell_{2j}}^s(s-\tau) \langle \tau \rangle^{-1} \big(\log\langle \tau \rangle\big)^{-(q-1)-b_j q} \left(\log\left(\tfrac{\tau}{\ell_{2j}}\right)\right)^{a_j q}\, d\tau \notag \\
&\geqslant 2^{-2} K C_j^q \big(\log\langle s \rangle\big)^{-(q-1)-b_j q} \langle s \rangle^{-1}\int_{\ell_{2j}}^s\tfrac{s-\tau}  {\tau} \left(\log\left(\tfrac{\tau}{\ell_{2j}}\right)\right)^{a_j q}\, d\tau. \label{estimate mathcalV inter 3} 
\end{align} Then,
\begin{align}
\int_{\ell_{2j}}^s \tfrac{s-\tau}{\tau} \left(\log\left(\tfrac{\tau}{\ell_{2j}}\right)\right)^{a_j q}\, d\tau 
&=  (a_j q+1)^{-1} \int_{\ell_{2j}}^s \left(\log\left(\tfrac{\tau}{\ell_{2j}}\right)\right)^{a_j q+1}\, d\tau \notag \\
&\geqslant (a_j q+1)^{-1} \int_{\tfrac{\ell_{2j} s}{\ell_{2j+1}}}^s \left(\log\left(\tfrac{\tau}{\ell_{2j}}\right)\right)^{a_j q+1} \, d\tau \notag \\ 
&\geqslant   (a_j q+1)^{-1}  \left(1-\tfrac{\ell_{2j}}{\ell_{2j+1}}\right) s \left(\log\left(\tfrac{s}{\ell_{2j+1}}\right)\right)^{a_j q+1} \notag \\
&\geqslant 2^{-(2j+5)}  (a_j q+1)^{-1}  \langle s\rangle \left(\log\left(\tfrac{s}{\ell_{2j+1}}\right)\right)^{a_j q+1} . \label{estimate mathcalV inter 4}
\end{align}
 A crucial difference with respect to the case $p>q$ is that we can increase the power for  the logarithmic term,  using integration by parts, even in this first stage of the inductive step.
Plugging \eqref{estimate mathcalV inter 4} in \eqref{estimate mathcalV inter 3}, we get
\begin{align*}
\mathcal{V}(s) &\geqslant 2^{-(2j+7)} K C_j^q  (a_j q+1)^{-1}  \big(\log\langle s \rangle\big)^{-(q-1)-b_j q}  \left(\log\left(\tfrac{s}{\ell_{2j+1}}\right)\right)^{a_j q+1}.
\end{align*} 
Then, we combine the previous lower bound for $\mathcal{V}(s)$ with \eqref{II2 crit}, so that for $t\geqslant \ell_{2j+2}$ it follows
\begin{align}
 \mathcal{U}(t) & \geqslant  2^{-(2j+7)p} C K^p C_j^{pq} (a_j q+1)^{-p} \langle t\rangle^{-1} \int_{\ell_{2j+1}}^t(t-s) \langle s \rangle^{-1}  \big(\log\langle s \rangle\big)^{-(pq-1)-b_j pq} \left(\log\left(\tfrac{s}{\ell_{2j+1}}\right)\right)^{a_j pq+p}  \, ds \notag \\
&\geqslant  2^{-(2j+7)p} C K^p C_j^{pq} (a_j q+1)^{-p} \big(\log\langle t \rangle\big)^{-(pq-1)-b_j pq} \langle t\rangle^{-1} \int_{\ell_{2j+1}}^t(t-s) \langle s \rangle^{-1}   \left(\log\left(\tfrac{s}{\ell_{2j+1}}\right)\right)^{a_j pq+p} \, ds  \notag \\
&\geqslant  2^{-(2j+7)p-2} C K^p C_j^{pq} (a_j q+1)^{-p}  \big(\log\langle t \rangle\big)^{-(pq-1)-b_j pq} \langle t\rangle^{-1} \int_{\ell_{2j+1}}^t \tfrac{t-s}{ s}   \left(\log\left(\tfrac{s}{\ell_{2j+1}}\right)\right)^{a_j pq+p} \, ds. \label{estimate mathcalU inter 3} 
\end{align} We use again integration by parts. Thus,
\begin{align}
 \int_{\ell_{2j+1}}^t \tfrac{t-s}{ s}   \left(\log\left(\tfrac{s}{\ell_{2j+1}}\right)\right)^{a_j pq+p} \, ds
& = (a_j pq+p+1)^{-1}   \int_{\ell_{2j+1}}^t  \left(\log\left(\tfrac{s}{\ell_{2j+1}}\right)\right)^{a_j pq+p+1} \, ds \notag \\
& \geqslant   (a_j pq+p+1)^{-1}   \int_{\tfrac{\ell_{2j+1}t}{\ell_{2j+2}}}^t \left(\log\left(\tfrac{s}{\ell_{2j+1}}\right)\right)^{a_j pq+p+1} \, ds \notag \\
& \geqslant  (a_j pq+p+1)^{-1}   \left(1-\tfrac{\ell_{2j+1}}{\ell_{2j+2}}\right) t \left(\log \left(\tfrac{t}{\ell_{2j+2}}\right)\right)^{a_j pq+p+1} \notag \\
& \geqslant    2^{-2(j+3)}  (a_j pq+p+1)^{-1} \langle t \rangle \left(\log\left(\tfrac{t}{\ell_{2j+2}}\right)\right)^{a_j pq+p+1}. \label{estimate mathcalU inter 4} 
\end{align}
If we combine \eqref{estimate mathcalU inter 3} and \eqref{estimate mathcalU inter 4}, then, we arrive at 
\begin{align*}
\mathcal{U}(t) & \geqslant  2^{-(p+1)2j-7p-8} C K^p C_j^{pq} (a_j q+1)^{-p} (a_j pq+p+1)^{-1}  \big(\log\langle t \rangle\big)^{-(pq-1)-b_j pq}\left(\log\left(\tfrac{t}{\ell_{2j+2}}\right)\right)^{a_j pq+p+1}.
\end{align*}
  Putting 
\begin{align*} C_{j+1}&\doteq 2^{-(p+1)2j-7p-8} C K^p C_j^{pq} (a_j q+1)^{-p} (a_j pq+p+1)^{-1}, \\  a_{j+1}&\doteq a_j pq+p+1\ \ \mbox{ and} \  \ b_{j+1} \doteq (pq-1)+b_j pq,
\end{align*} from the last inequality we get \eqref{iter ineq crit case} for $j+1$ when $p=q$.

Let us write the expressions of $a_j$ and $b_j$, 
\begin{align}
a_j  & = a_{j-1} pq+p+1  = a_0(pq)^j +(p+1)\sum_{k=0}^{j-1} (pq)^k  = \big(1+\tfrac{p+1}{pq-1}\big)(pq)^j  -\tfrac{p+1}{pq-1}\, , \label{def aj crit case p=q}\\
b_j & = (pq-1)+b_{j-1} pq  = (pq-1)\sum_{k=0}^{j-1} (pq)^k+b_{0} (pq)^j = (pq)^j-1. \label{def bj crit case p=q}
\end{align} Therefore, 
\begin{align*}
a_{j-1} pq+p+1& =\tfrac{p(q+1)}{pq-1}(pq)^{j+1}-\tfrac{p+1}{pq-1}\leqslant \tfrac{p(q+1)(pq)}{pq-1}(pq)^j , \\
a_{j-1} q+1 &=\tfrac{pq(q+1)}{pq-1}(pq)^{j}-\tfrac{q+1}{pq-1}\leqslant \tfrac{pq(q+1)}{pq-1}(pq)^j ,
\end{align*} 
so that we have again \eqref{inter est 3} but now with $\Theta\doteq 2^{2(p+1)}(pq)^{p+1}$ and $M\doteq 2^{-5p-6}CK^p \tfrac{(pq-1)^{p+1}}{p(q+1)^{p+1}(pq)^{p+1}}$. 

\subsubsection*{Lower bound for $C_j$}

Let us derive now a lower bound for $C_j$, in which the dependence on $j$ can be more easily handled than in \eqref{inter est 3}. Applying the logarithmic function to both sides of \eqref{inter est 3} and iterating the obtained relation, we get
\begin{align*}
\log C_j & \geqslant (pq) \log C_{j-1} -j\log \Theta+\log M \\
& \geqslant (pq)^2  \log C_{j-2}  -\big(j+(j-1)(pq)\big)\log \Theta+(1+pq)\log M  
\\
& \geqslant \cdots  \geqslant (pq)^j  \log C_{0}  -\sum_{k=0}^{j-1}(j-k)(pq)^k\log \Theta+\sum_{k=0}^{j-1} (pq)^k \log M  \\
 & = (pq)^j  \log C_{0}  - (pq)^j\sum_{k=1}^{j}\frac{k}{(pq)^k}\log \Theta+\frac{(pq)^j-1}{pq-1} \log M \\
 & = (pq)^j  \bigg( \log C_{0}  - S_j \log\Theta+\frac{\log M }{pq-1} \bigg) -\frac{\log M  }{pq-1} ,
\end{align*} where $S_j\doteq\sum_{k=1}^{j}\frac{k}{(pq)^k}$. By the ratio test it follows immediately that $\{S_j\}_{j\geqslant 1}$ is the sequence of partial sums of a convergent series. Therefore, if we denote by $S$ the limit of this sequence, then, since $S_j\uparrow S$ as $j\to \infty$ we may estimate 
\begin{align}\label{lower bound Cj crit case}
 C_j & \geqslant   M^{-(pq-1)} \exp \Big( (pq)^j   \log \Big( C_{0} \Theta^{-S} M^{pq-1} \Big) \Big).
\end{align}


\subsection{Conclusion of the proof of Theorem \ref{Thm critical case}} \label{Subsection crit case: lifespan estimate}

In this section we complete the proof in the critical case $F(n,p,q)=0$. 
Summing up the results of the last section, from \eqref{def aj crit case}, \eqref{def bj crit case}, \eqref{def aj crit case p=q}, \eqref{def bj crit case p=q} it follows the validity of \eqref{iter ineq crit case} with 
\begin{align}\label{def aj and bj both cases}
a_j=A(pq)^j+1-A, \qquad b_j=B(pq)^j-B,
\end{align} where
\begin{align*}
A & \doteq \begin{cases} \frac{pq}{pq-1} & \mbox{if} \ \ p>q, \\ 1+\frac{p+1}{pq-1} & \mbox{if} \ \ p=q,\end{cases} \qquad B  \doteq \begin{cases} \frac{p(q-1)}{pq-1} & \mbox{if} \ \ p>q, \\ 1  & \mbox{if} \ \ p=q.\end{cases}
\end{align*}

Combining \eqref{iter ineq crit case}, \eqref{lower bound Cj crit case} and  \eqref{def aj and bj both cases}, we arrive at 
\begin{align*}
\mathcal{U}(t) & \geqslant M^{-(pq-1)} \exp \Big( (pq)^j   \log \big(  C_{0} \Theta^{-S} M^{pq-1} \big) \Big) (\log\langle t\rangle)^{-B(pq)^j+B} \bigg(\log\bigg(\frac{t}{\ell_{2j}}\bigg)\bigg)^{A(pq)^j+1-A} \\
& \geqslant M^{-(pq-1)} \exp \Big( (pq)^j     \log \big( C_{0} \Theta^{-S} M^{pq-1} \big) \Big) (\log\langle t\rangle)^{-B(pq)^j+B} \big(\log\big(\tfrac{t}{2}\big)\big)^{A(pq)^j+1-A} \\
 & \geqslant M^{-(pq-1)} \exp \Big( (pq)^j   \log \Big(  C_{0} \Theta^{-S} M^{pq-1} (\log\langle t\rangle)^{-B} \big(\log\big(\tfrac{t}{2}\big)\big)^{A} \Big) \Big) (\log\langle t\rangle)^{B} \big(\log\big(\tfrac{t}{2}\big)\big)^{1-A}
\end{align*} for any $t\geqslant 2$.
 Since $\log(3+t)\leqslant \log(2t)\leqslant 2 \log t$ and $\log(\frac{t}{2})\geqslant\frac{1}{2}\log t$ for any $t\geqslant 4$, from the last estimate we may derive the following estimate for $t\geqslant 4$:
\begin{align}
\mathcal{U}(t) &  \geqslant M^{-(pq-1)} \exp \Big( (pq)^j   \log \Big( 2^{-B-A} C_{0} \Theta^{-S} M^{pq-1} (\log t)^{A-B} \Big) \Big) (\log\langle t\rangle)^{B} \big(\log\big(\tfrac{t}{2}\big)\big)^{1-A}. \label{inter est 4}
\end{align}
Let us consider separately the cases $p>q$ and $p=q$.

\subsubsection*{Case $p>q$} In this case $C_0= \widetilde{C} \varepsilon^{pq}$. Hence, \eqref{inter est 4} implies
\begin{align}
\mathcal{U}(t) &  \geqslant M^{-(pq-1)} \exp \Big( (pq)^j   \log \Big( E \varepsilon^{pq}  (\log t)^{\frac{p}{pq-1}} \Big) \Big) (\log\langle t\rangle)^{B} \big(\log\big(\tfrac{t}{2}\big)\big)^{1-A}, \label{final lower bound mathcal U}
\end{align} where $E\doteq 2^{-B-A} \widetilde{C} \, \Theta^{-S} M^{pq-1}$. Let us denote $K(t)\doteq \log \Big( E \varepsilon^{pq}  (\log t)^{\frac{p}{pq-1}} \Big)$. 

We can choose $\varepsilon_0=\varepsilon_0(u_0,u_1,v_0,v_1,n,p,q,b_1,b_2,R)>0$ so small that
\begin{align*}
\exp\Big(E^{-\frac{pq-1}{p}} \varepsilon_0^{-q(pq-1)}\Big)\geqslant 4.
\end{align*} Therefore, for any $\varepsilon\in(0,\varepsilon_0]$ and any $t\geqslant \exp\Big(E^{-\frac{pq-1}{p}} \varepsilon^{-q(pq-1)}\Big)$ we get $t \geqslant 4$ $K(t)>0$ and, consequently, taking the limit in \eqref{final lower bound mathcal U} as $j\to \infty$ we find that $\mathcal{U}(t)$ is not finite.
 Also, we proved the upper bound estimate  for the lifespan $T\leqslant \exp\Big(E^{-\frac{pq-1}{p}} \varepsilon^{-q(pq-1)}\Big)$.

\subsubsection*{Case $p=q$} In this case $C_0= \widetilde{C} \varepsilon^{p}$. Therefore, \eqref{inter est 4} yields
\begin{align*}
\mathcal{U}(t) &  \geqslant M^{-(pq-1)} \exp \Big( (pq)^j   \log \Big( E \varepsilon^{p}  (\log t)^{\frac{1}{p-1}} \Big) \Big) (\log\langle t\rangle)^{B} \big(\log\big(\tfrac{t}{2}\big)\big)^{1-A}.
\end{align*} 
Repeating the same steps as in the first case, we get the upper bound estimate for the lifespan
$$T\leqslant \exp\Big(E^{-(p-1)} \varepsilon^{-p(p-1)}\Big).$$
This conclude the proof of Theorem \ref{Thm critical case}.

\section*{Acknowledgments}

The first author is member of the Gruppo Nazionale per L'Analisi Matematica, la Probabilit\`{a} e le loro Applicazioni (GNAMPA) of the Instituto Nazionale di Alta Matematica (INdAM). This paper was written partially during the stay of the first author at Tohoku University in 2018. He would like to thank the Mathematical Department of Tohoku University for the hospitality and the excellent working conditions during this period. 
The second author is partially supported by the Grant-in-Aid for Scientific Research (B)(No.18H01132).

\vspace*{0.5cm}




\begin{thebibliography}{00}


\bibitem{AKT00} R. Agemi, Y.  Kurokawa, H. Takamura,
	\newblock{Critical curve for $p$-$q$ systems of nonlinear wave equations in three space dimensions,}
	\newblock{J. Differential Equations} {\bf 167}(1) (2000), 87--133. 
	%
	\bibitem{DelS97} D. Del Santo,
	\newblock{Global existence and blow-up for a hyperbolic system in three space dimensions,}
	\newblock{Rend. Istit. Mat. Univ. Trieste} {\bf 29}(1-2) (1997), 115--140.
	%
	\bibitem{DM} D. Del Santo, E. Mitidieri,
	\newblock{Blow-up of solutions of a hyperbolic system: The critical case,}
	\newblock{Differential Equations}, {\bf 34}(9)  (1998), 1157--1163.
	%
	\bibitem{DGM} D. Del Santo, V. Georgiev, E. Mitidieri,
	\newblock{Global existence of the solutions and formation of singularities for a class of hyperbolic systems,}
	\newblock{in: F. Colombini, N. Lerner (Eds.) Geometrical Optics and Related Topics. Progress in Nonlinear Differential Equations and Their Applications, vol 32. Birkh\"auser, Boston, MA,} (1997), https://doi.org/10.1007/978-1-4612-2014-5$\_$7.
	%
	%
\bibitem{Geo97} V. Georgiev, H. Lindblad, C.D. Sogge,
\newblock {Weighted Strichartz estimates and global existence for semi-linear wave equations,}
\newblock { Amer. J. Math.} {\bf 119}(6) (1997), 1291--1319.

%
	\bibitem{GTZ06} V. Georgiev, H. Takamura, Y. Zhou, 
	\newblock {The lifespan of solutions to nonlinear systems of a high-dimensional wave equation,}
	\newblock {Nonlinear Anal.} {\bf 64}(10) (2006), 2215--2250.
	%
\bibitem{Glas81B} R. T. Glassey,
\newblock {Finite-time blow-up for solutions of nonlinear wave equations,}
\newblock { Math Z.} {\bf 177}(3) (1981), 323--340.
%
\bibitem{Glas81} R. T. Glassey,
\newblock {Existence in the large for $\square u = F(u)$ in two space dimensions,}
\newblock { Math Z.} {\bf 178}(2) (1981), 233--261.
%
%
	\bibitem{IS17} M. Ikeda, M. Sobajima,
	\newblock{Life-span of solutions to semilinear wave equation with time-dependent critical damping for specially localized initial data,}
	\newblock{Math. Ann.} (2018),   https://doi.org/10.1007/s00208-018-1664-1.
	%
	\bibitem{IS18} M. Ikeda, M. Sobajima,
	\newblock{Upper bound for lifespan of solutions to certain semilinear parabolic, dispersive and hyperbolic equations via a unified test function method,} 
	\newblock{preprint, arXiv:1710.06780v2, 2018.} 
	%
	\bibitem{ISW18} M. Ikeda, M. Sobajima, K. Wakasa,
	\newblock{Blow-up phenomena of semilinear wave equations and their weakly coupled systems,} 
	\newblock{preprint, arXiv:1807.03937v1, 2018.} 
	%
	%
\bibitem{Jiao03} H. Jiao, Z. Zhou,
\newblock { An elementary proof of the blow-up for semilinear wave equation in high space dimensions,}
\newblock{ J. Differential Equations} {\bf 189}(2) (2003),  355--365.
	%

\bibitem{John79} F. John,
\newblock {Blow-up of solutions of nonlinear wave equations in three space dimensions,}
\newblock{ Manuscripta Math.} {\bf 28}(1-3) (1979), 235--268.
	%
\bibitem{Kato80} T. Kato,
\newblock{ Blow-up of solutions of some nonlinear hyperbolic equations,} \newblock{ Comm. Pure Appl. Math.} {\bf 33}(4) (1980), 501--505.
	%
	\bibitem{Kur05} Y. Kurokawa,
	\newblock{The lifespan of radially symmetric solutions to nonlinear systems of odd dimensional wave equations,}
	\newblock{Tsukuba J. Math.} {\bf 60}(7) (2005), 1239--1275.
	%
	\bibitem{KT03} Y. Kurokawa, H. Takamura,
	\newblock{A weighted pointwise estimate for two dimensional wave equations and its applications to nonlinear systems,}
	\newblock{Tsukuba J. Math.} {\bf 27}(2) (2003), 417--448.
	%
	\bibitem{KTW12} Y. Kurokawa, H. Takamura, K. Wakasa,
	\newblock{The blow-up and lifespan of solutions to systems of semilinear wave equation with critical exponents in high dimensions,}
	\newblock{Differential Integral Equations} {\bf 25}(3-4) (2012), 363--382.
	%
	\bibitem{LT18Scatt} N. A. Lai, H. Takamura,
	\newblock{Blow-up for semilinear damped wave equations with subcritical exponent in the scattering case,} \newblock{ Nonlinear Anal.} {\bf 168} (2018), 222--237.
	%
	\bibitem{LT18Glass} N.A. Lai, H. Takamura,
	\newblock{Nonexistence of global solutions of nonlinear wave equations with weak time-dependent damping related to Glassey's conjecture,} \newblock{ Differential and Integral Equations} {\bf 32}(1-2) (2019), 37--48.
	%
	\bibitem{LT18ComNon} N.A. Lai, H. Takamura,
	\newblock{Nonexistence of global solutions of wave equations with weak time-dependent damping and combined nonlinearity,} \newblock{ Nonlinear Anal. Real World Appl.} {\bf 45} (2019), 83--96.
	%
	\bibitem{LTW17} N.A. Lai, H. Takamura, K. Wakasa, 
	\newblock{Blow-up for semilinear wave equations with the scale invariant damping and super-Fujita exponent,}
	\newblock{J. Differential Equations} {\bf 263}(9) (2017), 5377--5394.
	%
\bibitem{Lin90} H. Lindblad,
\newblock{ Blow-up for solutions of $\square u =|u|^p$ with small initial data.}
\newblock{ Comm. Partial Differential Equations} {\bf 15}(6) (1990),  757--821.
%
\bibitem{LinSog95} H. Lindblad, C. Sogge,
\newblock{On existence and scattering with minimal regularity for semilinear
wave equations.}
\newblock{ J. Funct. Anal.} {\bf 130}(2) (1995), 357--426.
%
\bibitem{LinSog96} H. Lindblad, C. Sogge,
\newblock{ Long-time existence for small amplitude semilinear wave
equations.}
\newblock{ Amer. J. Math.} {\bf 118}(5) (1996), 1047--1135.
	%
\bibitem{Pal19} A. Palmieri,
	\newblock{A note on a conjecture for the critical curve of a weakly coupled system of semilinear wave equations with scale-invariant lower order terms,}  preprint, arXiv:1812.06588v1, 2018.
%
\bibitem{PalTak19der} A. Palmieri, H. Takamura,
\newblock{Nonexistence of global solutions for a weakly coupled system of semilinear damped wave equations of derivative type in the scattering case,} \newblock{in preparation.}
%
\bibitem{PalTak19mix} A. Palmieri, H. Takamura,
\newblock{Nonexistence of global solutions for a weakly coupled system of semilinear damped wave equations in the scattering case with mixed nonlinear terms,} \newblock{in preparation.}
%
	\bibitem{PT18} A. Palmieri, Z. Tu,
	\newblock{Lifespan of semilinear wave equation with scale invariant dissipation and mass and sub-Strauss power nonlinearity,} \newblock{J. Math. Anal. Appl.} (2018), https://doi: 10.1016/j.jmaa.2018.10.015.
	%
	\bibitem{Scha85} J. Schaeffer,
\newblock {The equation $u_{tt}-\Delta u = |u|^p$ for the critical value of $p$,}
\newblock { Proc. Roy. Soc. Edinburgh Sect. A.} {\bf 101}(1-2) (1985), 31--44.
%
\bibitem{Sid84} T.C. Sideris,
\newblock { Nonexistence of global solutions to semilinear wave equations in high dimensions,}
\newblock { J. Differential Equations} {\bf 52}(3) (1984), 378--406.
%
\bibitem{Str81} W.A. Strauss,
\newblock{Nonlinear scattering theory at low energy,}
\newblock{ J. Funct. Anal.} {\bf 41}(1) (1981), 110--133. 
%
%
\bibitem{TakWak11} H. Takamura, K. Wakasa,
\newblock{The sharp upper bound of the lifespan of solutions to critical semilinear wave equations in high dimensions,}
\newblock{ J. Differential Equations} {\bf 251}(4-5) (2011), 1157--1171. 
%
\bibitem{Tat01} D. Tataru,
\newblock Strichartz estimates in the hyperbolic space and global existence for the semilinear wave equation,
\newblock { Trans. Amer. Math. Soc.} {\bf 353}(2) (2001), 795--807.	
	%
	\bibitem{TL1709} Z. Tu, J. Lin, 
	\newblock{A note on the blowup of scale invariant damping wave equation with sub-Strauss exponent,} 
	\newblock{preprint, arXiv:1709.00866v2, 2017.} 
	%
	\bibitem{WakYor18} K. Wakasa, B. Yordanov,
	\newblock{Blow-up of solutions to critical semilinear wave equations with variable coefficients,}
	\newblock{J. Differential Equations} (2018),
\newblock{https://doi.org/10.1016/j.jde.2018.10.028} 
	%
	\bibitem{WakYor18damp} K. Wakasa, B. Yordanov,
	\newblock{On the blow-up for critical semilinear wave equations with damping in the scattering case,}
	\newblock{Nonlinear Anal.} {\bf 180} (2019) 67--74. 
	%
	\bibitem{YZ06} B.T. Yordanov, Q.S. Zhang,
	\newblock{Finite time blow up for critical wave equations in high dimensions,}
	\newblock{J. Funct. Anal.} {\bf 231}(2) (2006), 361--374.
	%
\bibitem{Zhou92} Y. Zhou,
\newblock{Blow up of classical solutions to $\square u=|u|^{1+\alpha}$ in three space dimensions,}
\newblock{ J. Partial Differential Equations} {\bf 5}(3) (1992), 21--32.
%
\bibitem{Zhou93} Y. Zhou,
\newblock{ Life span of classical solutions to $\square u= |u|^p$ in two space dimensions,}
\newblock{ Chinese Ann. Math. Ser. B} {\bf 14}(2) (1993), 225--236. 

	%
	\bibitem{Zhou07} Y. Zhou,
	\newblock{Blow up of solutions to semilinear wave equations with critical exponent in high dimensions,} 
	\newblock{Chin. Ann. Math. Ser. B} {\bf 28}(2) (2007), 205--212.
	%
	\bibitem{ZH14} Y. Zhou, W. Han,
	\newblock{ Life-span of solutions to critical semilinear wave equations,}\newblock{ Comm. Partial Differential Equations} {\bf 39}(3) (2014), 439-451.
	%
	
	
	





















\end{thebibliography}


\end{document}